\documentclass[12pt]{article}
\usepackage[latin1]{inputenc}
\usepackage[T1]{fontenc}
\usepackage{amsmath}
\usepackage{amsfonts}
\usepackage{amsthm}
\theoremstyle{plain}
\newtheorem{theorem}{Theorem}[section]
\newtheorem{coro}{Corollary}[section]
\newtheorem{lemma}{Lemma}[section]

\newtheorem{propo}{Proposition}[section]
\theoremstyle{definition}
\newtheorem{defi}{Definition}[section]

\newtheorem{remark}{Remark}[section]

\newcommand{\R}{\mathbb{R}}
\newcommand{\E}{\mathbb{E}}
\newcommand{\Var}{\operatorname{Var}}
\newcommand{\PP}{\mathbb{P}}

\title{Limit theorems for the sample autocovariance of a continuous-time moving average process with long memory\footnote{Supported by Deutsche Forschungsgemeinschaft Grant LI 1026/4-2}}
\author{
Felix Spangenberg\thanks{Institut f\"ur Mathematische Stochastik, TU
Braunschweig,  Pockelsstra{\ss}e 14, D-38106 Braunschweig, Germany
\texttt{f.spangenberg@tu-bs.de}}}
\begin{document}
\maketitle

\begin{abstract}We examine the asymptotic behaviour of the sample autocovariance in a continuous-time moving average model with long-range dependence. We show that it is either asymptotically Rosenblatt distributed or stable distributed. This shows that results by Horváth and Kokoszka \cite{HK} for discrete-time moving average processes with long memory also hold for continuous-time moving average processes.
\end{abstract}

\section{Introduction}
Let $(Z_t)_{t \in \mathbb{Z}}$ be an i.i.d. sequence of real random variables with $\E[Z_0]=0$ and $\E[Z_0^2]=\sigma^2<\infty$. Let $(\psi_j)_{j \in \mathbb{N}_0}$ be a square summable real sequence. Then a one-sided discrete-time moving average process of infinite order $(X_t)_{t \in \mathbb{Z}}$ is defined by
\[X_t := \sum_{j = 0}^{\infty} \psi_{j} Z_{t-j}, \quad t \in \mathbb{Z},\]
where the limit exists in the $L^2$-sense and here as a sum of independent summands almost surely as well.
The autocovariances of the process are given by
\[\gamma(h):= \operatorname{Cov}(X_t,X_{t+h})= \sigma^2 \sum_{j =0}^{\infty} \psi_{j} \psi_{j+|h|}, \quad h \in \mathbb{Z}.\]
A canonical estimator for the autocovariances is given by
\[\hat{\gamma}_N(h):= \frac{1}{N} \sum_{t=1}^{N} X_t X_{t+|h|}, \quad h \in \mathbb{Z}.\]
Horváth and Kokoszka $\cite{HK}$ examined the asymptotic behaviour of this estimator under the assumption that $\psi_{j} = j^{d-1} l(j)$, with $d \in (0,\frac{1}{2})$, $l(j) \to C_d>0$ as $j \to \infty$ and that either $\E[Z_0^4]<\infty$ or that $Z_0$ is regularly varying with index $\alpha \in (2,4)$.
In the case $d \in (0,\frac{1}{4})$, they showed that the estimator is asymptotically normal distributed. In the case with $d \in (\frac{1}{4},\frac{1}{2})$ and finite fourth moments, the estimator is asymptotically Rosenblatt distributed and in the case with regularly varying noise, the estimator is asymptotically Rosenblatt distributed when $d>\frac{1}{\alpha}$ and asymptotically stable distributed when $d<\frac{1}{\alpha}$.

Let $(L_t)_{t \in \mathbb{R}}$ be a two-sided Lévy process with $\E[L_1]=0$ and $\E[L_1^2]=\sigma^2 \in (0,\infty)$. Let $f$ be a real-valued function in $L^2(\R)$. Then  a continuous-time moving average process $(X_t)_{t \in \R}$ is defined by
\begin{equation} X_t := \int_{-\infty}^{\infty} f(t-s) \, dL_s, \quad t \in \R, \label{CMAdefinition}\end{equation}
where the integral is defined in the $L^2$-sense of stochastic integrals. This process is easily seen to be strictly and weakly stationary. The autocovariance can be easily seen by Itô's isometry to be
\[\gamma(h)= \sigma^2 \int_{-\infty}^{\infty} f(s)f(s+h) \, ds, \quad h \in \R.\]
Cohen and Lindner \cite{CohenLindner} investigated the asymptotic behaviour of the sample mean, the sample autocovariance and the sample autocorrelation for continuous-time moving average processes under certain conditions. Among other things, they showed that the sample autocovariance is asymptotically normal distributed in their case.

The aim of this article is to derive the asymptotic behaviour of the sample autocovariance in the continuous-time case under assumptions similar to those in the article by Horváth and Kokoszka \cite{HK}. To this end, we assume that $f(t)=0$ for $t \leq 0$, $f$ is bounded and $f(t) \sim C_d t^{d-1}$ as $t \to \infty$ with $d \in (0,\frac{1}{2})$ and $C_d>0$. 

A class of discrete-time processes that are of the above form is given by ARFIMA processes, see for example Chapter 7 in Giraitis et al. \cite{GKS}. A class of continuous-time processes of the above form is given by FICARMA processes, which go back to Peter Brockwell, see \cite{Brockwell} and \cite{BrockwellMarquardt}.

In our article, we stick to the notation of Horváth and Kokoszka \cite{HK}.  They define the Rosenblatt process $(U_d(t))_{t \in \R}$ by \begin{equation} U_d(t) := 2 \int_{x_1 < x_2 < t} [ \int_0^t (v-x_1)_+^{d-1} (v-x_2)_+^{d-1} \, dv] \, W(dx_1)W(dx_2), \quad t \in \R, \label{Rosenblattprocess} \end{equation}
where $W$ is a standard Gaussian random measure on $\R$, i.e. standard Brownian motion.
Note that this definition depends on $d$. We call the distribution of $U_d(1)$ the Rosenblatt distribution.  For an introduction to multiple Wiener integrals, see Chapter 14 in \cite{GKS}.

We show in the second section under the assumption that $\E[L_1^4]<\infty$ and $d \in (\frac{1}{4},\frac{1}{2})$, that the sample autocovariance is asymptotically Rosenblatt distributed. This is in contrast to the case when $\E[L_1^4]<\infty$ and $d \in (0,\frac{1}{4})$ in which the sample autocovariance function is asymptotically normal distributed, as follows easily from results of Cohen and Lindner \cite{CohenLindner} and is shortly discussed in section 4. We show in section 3 under the assumption that $L_1$ is regularly varying with index $\alpha \in (2,4)$, that it is asymptotically either stable or Rosenblatt distributed if $d \in (0,\frac{1}{2}) \setminus \{\frac{1}{\alpha}\}$. In the case with regularly varying $L_1$, we have to restrict ourselves to symmetric Lévy processes, but we believe that this assumption is not too severe as we already assumed that its expectation vanishes. We need the symmetry in Remark \ref{replacingcentringsequence}. In section 4, we further discuss briefly that FICARMA processes satisfy the assumptions on the kernel function and that our results can also be applied to calculate the asymptotics of the sample autocorrelation and the asymptotics for an estimator for $d$ for fractional Lévy noise.

We conclude the introduction with some (notational) remarks. From the assumptions on $f$, we can conclude that there is a constant $K>0$ such that
\begin{equation} |f(t)|\leq K \max(1,t^{d-1}), \label{basicinequalitywithK}\end{equation}
which we use throughout the article. Further, we assume that we are given a probability space $(\Omega,\mathcal{A}, \PP)$. By $\E$ and $\Var$, we denote the expectation and variance with respect to $\PP$. By $L^1(\PP)$ and $L^2(\PP)$, we denote the Banach spaces of integrable and square integrable random variables, by $L^2(\R^d)$, we denote the space of square integrable functions with respect to the Lebesgue measure. By $\left\lfloor x \right\rfloor$ for $x \in \R$, we denote the largest integer that is not larger than $x$. We define $x_+^{d-1}:= x^{d-1}1_{(0,\infty)}(x)$ for $x \in \R$. By $g$, $G$, $u$, $v$ and $w$ we will denote auxiliary functions. We call a series unconditionally convergent, if its limit does not depend on the order of summation. By $h \in \mathbb{N}_0$, we denote the lag of the autocovariance function. Since the autocovariance function is symmetric, it suffices to assume $h \in \mathbb{N}_0$. We set $\varepsilon := \frac{1}{m}$ with $m \in \mathbb{N}$. Finally, when we use $o(N^d)$ and $O(N^d)$, we mean the asymptotics as $N \to \infty$.

\section{Theorem for finite fourth moments}
In this section we assume that the Lévy process $(L_t)_{t \in \R}$ has finite fourth moments and that $d \in (\frac{1}{4},\frac{1}{2})$.
Define the sample and the actual autocovariance function of the process $(X_t)_{t \in \R}$ defined in (\ref{CMAdefinition}) by \[\hat{\gamma}_N(h):=\frac{1}{N} \sum_{t=1}^{N} X_{t}X_{t+h}, \quad h \in \mathbb{N}_0\] and \[\gamma(h):= \operatorname{Cov}(X_t,X_{t+h}) =\sigma^2 \int_{-\infty}^{\infty} f(t)f(t+h) \, dt, \quad h \in \mathbb{N}_0.\]  We show that the sample autocovariance is then asymptotically Rosenblatt distributed. This is the statement of the following theorem, which parallels Theorem 3.3 (b) in \cite{HK}.

\begin{theorem}\label{Theoremfourthmoment}
Let $(L_t)_{t \in \mathbb{R}}$ be a two-sided Lévy process with $\E[L_1]=0$, $\Var(L_1)=\sigma^2 \in (0,\infty)$ and $\E[L_1^4] <\infty$. Define a continuous-time moving average process $(X_t)_{t \in \R}$ by \[X_t:=\int_{-\infty}^{\infty} f(t-s)  \, dL_s, \quad t \in \R,\] where we assume that $f(t)=0$ for $t \leq 0$, $f$ is bounded and $f(t) \sim C_d t^{d-1}$ as $t \to \infty$ with $d \in (\frac{1}{4},\frac{1}{2})$ and $C_d>0$.
Then
\[N^{1-2d} (\hat{\gamma}_N(0)-\gamma(0),\ldots, \hat{\gamma}_N(H)-\gamma(H)) \stackrel{d}{\rightarrow} C_d^2 \sigma^2 U_d(1) (1,\ldots,1)\mbox{ as } N \to \infty,\]
where $U_d(1)$ is the marginal distribution of the Rosenblatt process at time $1$ defined in (\ref{Rosenblattprocess}).
\end{theorem}

The rest of this section is devoted to the proof of Theorem \ref{Theoremfourthmoment}.
Our proof consists of two parts. In the first part, we approximate $f$ by
\[f_m:= \sum_{k=0}^{\infty} f(\varepsilon k) 1_{[k\varepsilon,(k+1)\varepsilon)}, \label{fm}\]
 where $m \in \mathbb{N}$ and $\varepsilon:= \frac{1}{m}$. We then show that the \textit{non-diagonal terms} of the sample autocovariance function which we denote by $r_{N,h,\varepsilon}$ and are defined in (\ref{definitionrn}) converge towards the Rosenblatt distribution, which is the statement of Lemma \ref{nondiagonalparttoRosenblatt}. We then use Slutsky's Lemma in the second part to show the asymptotics of the sample autocovariance for general kernel functions.

Define
\begin{equation}
Z_i:=L_{\varepsilon i}-L_{\varepsilon(i-1)}, \quad i \in \mathbb{Z}. \label{increments}
\end{equation} Observe that $\E[Z_i]=0$ and $\E[Z_i^2]=\varepsilon \sigma^2$.  Define
\[X_t^{(m)}:=\int_{-\infty}^{\infty} f_m(t-s) \, dL_s = \sum_{k=0}^{\infty} f(\varepsilon k) (L_{t-\varepsilon k}-L_{t-\varepsilon(k+1)})=\sum_{k=0}^{\infty} f(\varepsilon k) Z_{mt-k}, \quad t \in \mathbb{Z},\]
where the two series converge as series of independent, hence orthogonal elements in the Hilbert space $L^2(\PP)$, unconditionally in $L^2(\PP)$.

We divide the summands of the sample autocovariance in diagonal terms and non-diagonal terms with respect to  products of $(Z_t)_{t \in \mathbb{Z}}$. To this end, we need the following technical lemma:

\begin{lemma}\label{technicallemma1}
Let $A,B$ and $C$ be random variables. Let $(A_u)_{u \in \mathbb{Z}}$ and $(B_v)_{v \in \mathbb{Z}}$ be sequences of random variables such that $A = \sum_{u \in \mathbb{Z}} A_u$ and $B = \sum_{v \in \mathbb{Z}}B_v$ as unconditional $L^2(\PP)$ limits. Assume further that $C = \sum_{u,v \in \mathbb{Z}} A_u B_v$ as an unconditional $L^1(\PP)$ limit. Then $AB=C$ almost surely.
\end{lemma}

\begin{proof}
Let $\delta>0$. By our assumptions, we find an $N \in \mathbb{N}$ such that $||A- \sum_{|u|\leq N} A_u||_2 < \delta$, $||B- \sum_{|v|\leq N} B_v||_2 < \delta$ and $||C- \sum_{|u|,|v|\leq N} A_u B_v||_1 < \delta$. We then obtain
\begin{eqnarray*}
||C-AB||_1 & \leq & ||C- \sum_{|u|,|v| \leq N} A_u B_v||_1 + ||\sum_{|u|,|v| \leq N} A_u B_v  - (\sum_{|u| \leq N} A_u) B||_1 + ||(\sum_{|u| \leq N} A_u) B - AB||_1\\
&\leq & \delta +  ||(\sum_{|u| \leq N} A_u) (\sum_{|v| \leq N}B_v  -  B)||_1 + ||(\sum_{|u| \leq N} A_u - A)B||_1\\
&\leq & \delta + ||\sum_{|u| \leq N} A_u||_2 ||\sum_{|v| \leq N}B_v  -  B||_2 + ||\sum_{|u| \leq N} A_u - A||_{2} ||B||_2\\
&\leq & \delta + (||A||_2 + \delta) \delta + \delta ||B||_2.
\end{eqnarray*}
\end{proof}

Now the following rearrangement is justified by Lemma \ref{technicallemma1}, once we have shown that the right-hand side converges unconditionally in $L^1(\PP)$. Let $h \in \mathbb{N}_0$. Then
\begin{eqnarray}
X_t^{(m)} X_{t+h}^{(m)} &=& (\sum_{i=0}^{\infty} f(\varepsilon i) Z_{mt-i}) (\sum_{j=0}^{\infty} f(\varepsilon j) Z_{m(t+h)-j}) \label{lefthandsidereaarrangement}\\
& = & \sum_{i=0}^{\infty} f(\varepsilon i)f(\varepsilon i+h) (Z_{mt-i})^2 \nonumber\\
&   & + \sum_{j \neq i + mh, i,j \in \mathbb{N}_0} f(\varepsilon i)f(\varepsilon j) Z_{mt-i}Z_{m(t+h)-j}. \label{righthandsidereaarrangement}
\end{eqnarray}
Observe that $f(x)$ vanishes for $x \leq 0$, hence the sum can also be taken over $i,j \in \mathbb{Z}$. Note that $\sum_{i=0}^{\infty} f(\varepsilon i)f(\varepsilon i+h) (Z_{mt-i})^2$ converges absolutely almost surely and unconditionally in $L^1(\PP)$ because $\sum_{i=0}^{\infty} f(\varepsilon i)f(\varepsilon i+h)$ is absolutely summable and $(Z_{mt-i})^2$ has finite expectation. 
Setting $k=mt-i$ and $k'=m(t+h)-j$, the last summand can be rewritten as
\[\sum_{k \neq k', k,k' \in \mathbb{Z}} f(t - \varepsilon k) f(t+h-\varepsilon k') Z_k Z_{k'}.\]
We decompose this series as
\begin{eqnarray}
&&\sum_{k \neq k', k,k' \in \mathbb{Z}} f(t - \varepsilon k) f(t+h-\varepsilon k') Z_k Z_{k'} \label{lefthandsideorthognalfamlies}\\
&=&\sum_{k < k'} f(t - \varepsilon k) f(t+h-\varepsilon k') Z_k Z_{k'}+\sum_{k > k'} f(t - \varepsilon k) f(t+h-\varepsilon k') Z_k Z_{k'}.\label{righthandsideorthognalfamlies}
\end{eqnarray}
Since $(Z_k)_{k \in \mathbb{Z}}$ is i.i.d. with expectation zero and finite variance, both $(Z_kZ_{k'})_{k<k'}$ and $(Z_kZ_{k'})_{k>k'}$ are families of orthogonal elements in $L^2(\PP)$  with constant variance and since $(f(t-\varepsilon k)f(t+h- \varepsilon k'))_{k \neq k'}$ is square summable by assumption, the two series in  (\ref{righthandsideorthognalfamlies})  converge unconditionally in $L^2(\PP)$ and thus in $L^1(\PP)$ as well, hence the series in (\ref{lefthandsideorthognalfamlies})  can be seen to converge unconditionally as well. Hence (\ref{righthandsidereaarrangement}) converges unconditionally in $L^1(\PP)$ and hence by Lemma \ref{technicallemma1} (\ref{lefthandsidereaarrangement}) and  (\ref{righthandsidereaarrangement}) are equal.
Now define
\begin{equation} \label{definitionrn} r_{N,h,\varepsilon} := \frac{1}{N} \sum_{t=1}^{N} \sum_{k \neq k'} f(t - \varepsilon k) f(t+h-\varepsilon k') Z_k Z_{k'},\quad h \in \mathbb{N}_0.\end{equation}
Thus $r_{N,h,\varepsilon}$ represents all non-diagonal terms of the sample autocovariance.

The following lemma is a generalisation of Lemma 5.5 in \cite{HK}. Note that our proof is somewhat easier as we refer to results in \cite{GKS}. In fact, we only need the result for the case $\varepsilon = 1$ in this section, which follows from Lemma 5.5 in \cite{HK}, but we need the result for general $\varepsilon>0$ in section 3 and we state it already here for convenience.
\begin{lemma} \label{nondiagonalparttoRosenblatt} Let all assumptions of Theorem \ref{Theoremfourthmoment} apart from $\E[L_1^4] <\infty$ be fulfilled. Let $H \in \mathbb{N}_0$. Then
\[N^{1-2d}(r_{N,0,\varepsilon},\ldots,r_{N,H,\varepsilon}) \stackrel{d}{\rightarrow} C_d^2 \sigma^2 U_d(1) (1,\ldots,1) \mbox{ as } N \to \infty.\]
\end{lemma}

\begin{proof}
Without loss of generality, we can assume that $\sigma^2=1$ and $C_d=1$.

Define
\begin{equation} g(x_1,x_2):=\int_0^1 (t-\varepsilon x_1)_+^{d-1} (t-\varepsilon x_2)_+^{d-1} \, dt. \label{definitiong}\end{equation}
One can show that $g \in L^2(\R^2)$ in a similar fashion to equation (14.3.38) on page 544 in Giraitis et al. \cite{GKS}.

We will use Propositions 14.3.2 and 14.3.3 in Giraitis et al. \cite{GKS}.
To this end, we define
\begin{equation} g_N(k,k')= N^{-2d} \sum_{t=1}^{N} f(t-k\varepsilon) f(t+h-\varepsilon k') \label{definitiongN} \end{equation}
for $k \neq k'$ and $g_N(k,k)=0$ with $k,k' \in \mathbb{Z}$
and
\begin{equation} \tilde{g}_N(x_1,x_2) = N g_N(\left\lfloor N x_1 \right\rfloor, \left\lfloor N x_2 \right\rfloor) \label{definitiongNtilde}, \quad x_1,x_2 \in \R.\end{equation}
Then $\frac{Z_k}{\varepsilon}$ corresponds to $\zeta_k$ and $N^{1-2d} r_{N,h,\varepsilon}$
to $\varepsilon Q_2(g_N)$ in the notation of Proposition 14.3.2 in \cite{GKS}.
According to the next lemma, $\tilde{g}_N\stackrel{L^2(\R^2)}{\longrightarrow}g$.
Note that by substitution, it is easy to see that
\[\int_0^1 (v - \varepsilon x_1)_+^{d-1} (v  - \varepsilon x_2)_+^{d-1} \, dv = \varepsilon^{2d-1} \int_0^{\frac{1}{\varepsilon}} (v- x_1)_+^{d-1} (v - x_2)_+^{d-1} \, dv.\]
 Hence by Proposition 14.3.3 in \cite{GKS}, $N^{1-2d}(r_{N,0,\varepsilon},\ldots,r_{N,H,\varepsilon}) \stackrel{d}{\rightarrow}   \varepsilon^{2d}U_d({\frac{1}{\varepsilon}})(1,...,1)$ as $N \to \infty$.
The Rosenblatt process $(U_d(t))_{t \in \mathbb{R}}$ is self-similar with index $2d$, see \cite{GKS}, p. 544, Proposition 14.3.7, hence $\varepsilon^{2d}U_d({\frac{1}{\varepsilon}})\stackrel{d}{=} U_d(1)$.
\end{proof}
The following lemma was needed in the proof of Lemma \ref{nondiagonalparttoRosenblatt}:
\begin{lemma}
With the definitions in equations (\ref{definitiong}), (\ref{definitiongN}), (\ref{definitiongNtilde}) and the assumption $C_d=1$,
\[\tilde{g}_N\stackrel{L^2(\R^2)}{\longrightarrow}g \mbox{ as }N \to \infty\] holds.
\end{lemma}
\begin{proof}
Expressing the sum as an integral of a step function, we see that
\begin{eqnarray*}
\tilde{g}_N(x_1,x_2) & = & N^{1-2d} \sum_{t=1}^{N} f(t-\varepsilon\left\lfloor N x_1 \right\rfloor) f(t+h-\varepsilon \left\lfloor N x_2 \right\rfloor)\\
& = &   N^{2-2d} \int_{0}^{1} f(\left\lfloor Nt \right\rfloor +1-\varepsilon\left\lfloor N x_1 \right\rfloor) f(\left\lfloor Nt \right\rfloor +1 +h-\varepsilon \left\lfloor N x_2 \right\rfloor) \, dt,
\end{eqnarray*}
for $x_1, x_2 \in \R$ with $x_1\neq x_2$ and $\tilde{g}_N(x,x)=0$.

1. For our further calculations, we need estimates for $f(\left\lfloor Nt \right\rfloor +1+h-\varepsilon\left\lfloor N x_2 \right\rfloor)$ with $t \in [0,1]$. We consider the three cases $t > \varepsilon x_2$, $t < \varepsilon x_2 - \frac{1+h+\varepsilon}{N}$ and $\varepsilon x_2 - \frac{1+h+\varepsilon}{N} \leq t \leq \varepsilon x_2$.

If $t > \varepsilon x_2$ and $t \in [0,1]$, then $\left\lfloor Nt \right \rfloor +1+h > \varepsilon \left \lfloor N x_2 \right \rfloor$ and hence
\[f(\left\lfloor Nt \right\rfloor +1+h-\varepsilon\left\lfloor N x_2 \right\rfloor)\leq K (\left\lfloor Nt \right\rfloor +1+h-\varepsilon\left\lfloor N x_2 \right\rfloor)^{d-1} \leq K N^{d-1} (t-\varepsilon x_2)^{d-1}.\]
If $t < \varepsilon x_2 - \frac{1+h+\varepsilon}{N}$ and $t \in [0,1]$, then \[\left \lfloor Nt \right \rfloor +1 +h < \varepsilon \left \lfloor N x_2 \right \rfloor\] and hence \[f(\left\lfloor Nt \right\rfloor +1+h-\varepsilon\left\lfloor N x_2 \right\rfloor)=0.\] Since $f$ is bounded by $K$ by our assumptions, we also have \[|f(\left\lfloor Nt \right\rfloor +1+h-\varepsilon\left\lfloor N x_2 \right\rfloor)|\leq K,\] especially in the third case $\varepsilon x_2 - \frac{1+h+\varepsilon}{N} \leq t \leq \varepsilon x_2$.

2. It suffices to show $\tilde{g}_N1_{\{x_1<x_2\}}\stackrel{L^2(\R^2)}{\longrightarrow}g1_{\{x_1<x_2\}}$ and $\tilde{g}_N1_{\{x_1>x_2\}}\stackrel{L^2(\R^2)}{\longrightarrow}g1_{\{x_1>x_2\}}$. We only show $\tilde{g}_N1_{\{x_1<x_2\}}\stackrel{L^2(\R^2)}{\longrightarrow}g1_{\{x_1<x_2\}}$, the other convergence follows in an analogous manner. For the first convergence, we further decompose the function $\tilde{g}_N1_{\{x_1<x_2\}}$ into three parts according to the cases in the last paragraph with respect to $x_2$. For $x_2>x_1$ we now have

\begin{eqnarray*}
\tilde{g}_N(x_1,x_2) & = &  N^{2-2d} \int_{[0,1] \cap (\varepsilon x_2,\infty)} f(\left\lfloor Nt \right\rfloor +1-\varepsilon\left\lfloor N x_1 \right\rfloor) f(\left\lfloor Nt \right\rfloor +1 +h-\varepsilon \left\lfloor N x_2 \right\rfloor) \, dt\\
&+ &  N^{2-2d} \int_{[0,1] \cap (-\infty,\varepsilon x_2 - \frac{1+h+\varepsilon}{N})} f(\left\lfloor Nt \right\rfloor +1-\varepsilon\left\lfloor N x_1 \right\rfloor) \underbrace{f(\left\lfloor Nt \right\rfloor +1 +h-\varepsilon \left\lfloor N x_2 \right\rfloor)}_{=0} \, dt\\
&+ &  N^{2-2d} \int_{[0,1] \cap [\varepsilon x_2 - \frac{1+h+\varepsilon}{N},\varepsilon x_2]} f(\left\lfloor Nt \right\rfloor +1-\varepsilon\left\lfloor N x_1 \right\rfloor) f(\left\lfloor Nt \right\rfloor +1 +h-\varepsilon \left\lfloor N x_2 \right\rfloor) \, dt\\
&=:& G_N^{(1)}(x_1,x_2)+ G_N^{(2)}(x_1,x_2)
\end{eqnarray*}
with
\[G_N^{(1)}(x_1,x_2):=N^{2-2d} \int_{[0,1] \cap (\varepsilon x_2,\infty)} f(\left\lfloor Nt \right\rfloor +1-\varepsilon\left\lfloor N x_1 \right\rfloor) f(\left\lfloor Nt \right\rfloor +1 +h-\varepsilon \left\lfloor N x_2 \right\rfloor) \, dt\]
and
\[G_N^{(2)}(x_1,x_2):= N^{2-2d} \int_{[0,1] \cap [\varepsilon x_2 - \frac{1+h+\varepsilon}{N},\varepsilon x_2]} f(\left\lfloor Nt \right\rfloor +1-\varepsilon\left\lfloor N x_1 \right\rfloor) f(\left\lfloor Nt \right\rfloor +1 +h-\varepsilon \left\lfloor N x_2 \right\rfloor) \, dt.\]

3.   Let us again assume $x_1<x_2$:

For $t \in [0,1] \cap (\varepsilon x_2,\infty)$ we also have $t>\varepsilon x_1$ since we assume $x_1<x_2$ and hence
\[f(\left\lfloor Nt \right\rfloor +1-\varepsilon\left\lfloor N x_1 \right\rfloor) \leq K N^{d-1} (t-\varepsilon x_1)^{d-1}\]
and
\[f(\left\lfloor Nt \right\rfloor +1+h-\varepsilon\left\lfloor N x_2 \right\rfloor) \leq K N^{d-1} (t-\varepsilon x_2)^{d-1}.\]
Hence we have
\[N^{2-2d} |f(\left\lfloor Nt \right\rfloor +1-\varepsilon\left\lfloor N x_1 \right\rfloor)f(\left\lfloor Nt \right\rfloor +1+h-\varepsilon\left\lfloor N x_2 \right\rfloor)|\leq K^2 (t-\varepsilon x_1)^{d-1}(t-\varepsilon x_2)^{d-1}\]
and hence
\[G_N^{(1)}(x_1,x_2) \leq K^2 g(x_1,x_2)= K^2\int_0^1 (t-\varepsilon x_1)^{d-1}(t-\varepsilon x_2)^{d-1} \, dt <\infty.\]

Because of $t>\varepsilon x_1$, $t>\varepsilon x_2$ we have $\left\lfloor Nt \right\rfloor +1-\varepsilon\left\lfloor N x_1 \right\rfloor \to \infty$ and $\left\lfloor Nt \right\rfloor +1+h-\varepsilon\left\lfloor N x_2 \right\rfloor \to \infty$ as $N \to \infty$ as well as $\frac{\left\lfloor Nt \right\rfloor +1-\varepsilon\left\lfloor N x_1 \right\rfloor}{N(t-\varepsilon x_1)} \to 1$ and $\frac{\left\lfloor Nt \right\rfloor +1+h-\varepsilon\left\lfloor N x_2 \right\rfloor}{N(t-\varepsilon x_2)} \to 1$.
Hence
\begin{eqnarray*}
&& \lim_{N \to \infty} N^{2-2d} f(\left\lfloor Nt \right\rfloor +1-\varepsilon\left\lfloor N x_1 \right\rfloor)f(\left\lfloor Nt \right\rfloor +1+h-\varepsilon\left\lfloor N x_2 \right\rfloor)\\ &=& \lim_{N \to \infty} N^{2-2d} (\left\lfloor Nt \right\rfloor +1-\varepsilon\left\lfloor N x_1 \right\rfloor)^{d-1} (\left\lfloor Nt \right\rfloor +1+h-\varepsilon\left\lfloor N x_2 \right\rfloor)^{d-1}\\
&=& (t-\varepsilon x_1)^{d-1} (t-\varepsilon x_2)^{d-1}.
\end{eqnarray*}
With Lebesgue's convergence theorem, we conclude
\begin{eqnarray*}\lim_{N \to \infty} G_N^{(1)} (x_1,x_2)&=& \int_{[0,1] \cap (\varepsilon x_2,\infty)} (t-\varepsilon x_1)_+^{d-1} (t-\varepsilon x_2)_+^{d-1} \, dt\\
&=& \int_{[0,1]} (t-\varepsilon x_1)_+^{d-1} (t-\varepsilon x_2)_+^{d-1} \, dt= g(x_1,x_2).
\end{eqnarray*}
Since $G_N^{(1)}(x_1,x_2)1_{\{x_1<x_2\}}$ is bounded by $K^2g$ which is square integrable and $G_N^{(1)}(x_1,x_2)1_{\{x_1<x_2\}}$ converges pointwise towards $g1_{\{x_1<x_2\}}$,  $G_N^{(1)}$ converges to $g1_{\{x_1<x_2\}}$ in $L^2(\R^2)$ by Lebesgue's theorem.

4. For the  $L^2$-convergence, it suffices to show that $G_N^{(1)}1_{\{x_1<x_2\}} \stackrel{L^2(\R^2)}{\longrightarrow} g1_{\{x_1<x_2\}}$ and $G_N^{(2)}1_{\{x_1<x_2\}} \stackrel{L^2(\R^2)}{\longrightarrow} 0$. That $G_N^{(1)}1_{\{x_1<x_2\}} \stackrel{L^2(\R^2)}{\longrightarrow} g1_{\{x_1<x_2\}}$ was already shown in the last paragraph.
Define $G^{(3)}_N:=G_N^{(2)}1_{\{0<x_2-x_1\leq 1\}}$ and $G^{(4)}_N:=G_N^{(2)}1_{\{x_2-x_1>1\}}$. We show $G_N^{(3)} \stackrel{L^2(\R^2)}{\longrightarrow} 0$ and $G_N^{(4)} \stackrel{L^2(\R^2)}{\longrightarrow} 0$.
Note that $G^{(2)}_N(x_1,x_2)$ vanishes for $x_2<0$ and $\varepsilon x_2 >3 \geq 1 + \frac{1+h+\varepsilon}{N}$ (which is true for $N$ sufficiently large).

5. Define \[A(\varepsilon):= \{t \in [0,1] \cap [\varepsilon x_2 - \frac{1+h+\varepsilon}{N}, \varepsilon x_2]: \left\lfloor Nt\right\rfloor +1+h - \varepsilon \left\lfloor N x_2 \right\rfloor>0 \}\] and
\[B(\varepsilon):= \{t \in [0,1] \cap [\varepsilon x_2 - \frac{1+h+\varepsilon}{N}, \varepsilon x_2]: \left\lfloor Nt\right\rfloor +1+h - \varepsilon \left\lfloor N x_2 \right\rfloor\leq 0 \}.\]
Then $f(\left\lfloor Nt\right\rfloor +1+h - \varepsilon \left\lfloor N x_2 \right\rfloor)$ vanishes for $t  \in B(\varepsilon)$. Further, denoting by $\lambda$ the one-dimensional Lebesgue measure, we see that $\lambda(A(\varepsilon)) \leq \frac{1+h+\varepsilon}{N}$.
For $t \in A(\varepsilon)$, we have \[\left\lfloor Nt\right\rfloor +1+h - \varepsilon \left\lfloor N x_2 \right\rfloor\geq \varepsilon\] and hence
\[\left\lfloor Nt\right\rfloor +1 - \varepsilon \left\lfloor N x_1 \right\rfloor \geq \varepsilon-h + \varepsilon(\left\lfloor N x_2 \right\rfloor - \left\lfloor N x_1\right\rfloor).\]
With this estimate and the fact that $f(x)$ is bounded by $K (\max(x,1))^{d-1}$, we obtain
\begin{eqnarray*}
|G_N^{(2)}(x_1,x_2)| & = &N^{2-2d} \int_{A(\varepsilon)} f(\left\lfloor Nt \right\rfloor +1-\varepsilon\left\lfloor N x_1 \right\rfloor) f(\left\lfloor Nt \right\rfloor +1 +h-\varepsilon \left\lfloor N x_2 \right\rfloor) \, dt\\
& \leq& \lambda (A(\varepsilon)) N^{2-2d} K^2 \max(\varepsilon-h + \varepsilon(\left\lfloor N x_2 \right\rfloor - \left\lfloor N x_1\right\rfloor),1)^{d-1}\\
& \leq & (1+h+\varepsilon) K^2 N^{1-2d} \max(\varepsilon-h + \varepsilon(\left\lfloor N x_2 \right\rfloor - \left\lfloor N x_1\right\rfloor),1)^{d-1}.
\end{eqnarray*}
One can see that
\[\lambda(\{x_1 \in [x_2-1,x_2]: \left\lfloor N x_2 \right\rfloor - \left\lfloor N x_1 \right\rfloor = i\}) \leq \frac{1}{N}\]
for $i \in \mathbb{N}_0$ and that $\varepsilon - h + \varepsilon i>1$ for $i > \frac{1+h}{\varepsilon}-1$. 
Hence
\begin{eqnarray*}
||G^{(3)}_N||_{L^2(\R^2)}^2 & = & \int_0^{\frac{3}{\varepsilon}} \int_{x_2-1}^{x_2} |G_N^{(2)}(x_1,x_2)|^2 \, dx_1 dx_2\\
& \leq &(1+h+\varepsilon)^2 K^4 N^{1-4d} \frac{3}{\varepsilon}  (\frac{h+1}{\varepsilon} +\sum_{i=\frac{h+1}{\varepsilon}}^{\infty} ((\varepsilon -h + \varepsilon i)_+^{2d-2}))\\
& \stackrel{N \to \infty}{\rightarrow}&0.
\end{eqnarray*}

6.
 Now let $t \in [0,1] \cap [\varepsilon x_2 - \frac{1+h+\varepsilon}{N}, \varepsilon x_2]$ and $x_2-x_1>1$. Then $t>\varepsilon x_1$ for $N>\frac{1+h+\varepsilon}{\varepsilon}$. 
Hence we have
\[N^{2-2d} |f(\left\lfloor Nt \right\rfloor +1-\varepsilon\left\lfloor N x_1 \right\rfloor)f(\left\lfloor Nt \right\rfloor +1+h-\varepsilon\left\lfloor N x_2 \right\rfloor)| \leq K (t-\varepsilon x_1)^{d-1} K N^{1-d}\]
and hence
\begin{eqnarray*}
|G_N^{(2)}(x_1,x_2)| & \leq &K^2 N^{1-d} \int_{[0,1] \cap [\varepsilon x_2 - \frac{1+h+\varepsilon}{N},\varepsilon x_2]}(t-\varepsilon x_1)^{d-1} \, dt \\
& \leq & (1+h+ \varepsilon) K^2 (\varepsilon x_2 - \frac{1+h+\varepsilon}{N} - \varepsilon x_1)^{d-1} N^{-d}.
\end{eqnarray*}
 Let $N$ be sufficiently large such that $\frac{1+h+\varepsilon}{N} < \frac{\varepsilon}{2}$. We then have
\begin{eqnarray*}
\int_{-\infty}^{x_2-1} |G_N^{(2)}(x_1,x_2)|^2 \, dx_1 & \leq & K^4 (1+h+\varepsilon)^2 N^{-2d} \int_{-\infty}^{x_2-1} (\varepsilon x_2 - \frac{1+h+\varepsilon}{N} - \varepsilon x_1)^{2d-2} \, dx_1\\
& \leq & K^4 (1+h+\varepsilon)^2 N^{-2d} \int_{-\infty}^{x_2-1} (\varepsilon x_2 - \frac{\varepsilon}{2} - \varepsilon x_1)^{2d-2} \, dx_1\\
& = & K^4 (1+h+\varepsilon)^2 N^{-2d} \varepsilon^{2d-2} \underbrace{\int_{\frac{1}{2}}^{\infty} v^{2d-2} \, dv}_{< \infty}\\
\end{eqnarray*}
and hence
\begin{eqnarray*}
||G^{(4)}_N||_{L^2(\R^2)}^2 & = & \int_0^{\frac{3}{\varepsilon}} \int_{-\infty}^{x_2-1} |G_N^{(2)}(x_1,x_2)|^2 \, dx_1 dx_2\\
& \leq & K^4 (1+h+\varepsilon)^2 N^{-2d} \varepsilon^{2d-2} \frac{3}{\varepsilon} \int_{\frac{1}{2}}^{\infty} v^{2d-2} \, dv\\
& \stackrel{N \to \infty }{\rightarrow} & 0.
\end{eqnarray*}

\end{proof}
Returning to the proof Theorem \ref{Theoremfourthmoment},  we define for the process $(X_t)_{t \in \R}$ similar random variables that correspond to the increments and the non-diagonal  parts of the autocovariance function of the process $(X_t^{(m)})_{t \in \R}$ defined in (\ref{increments}) and (\ref{definitionrn}). Note that the diagonal part $d_{N,h,\varepsilon}$ is only needed in section 3, which is defined in (\ref{definitiondn}).
Define
\begin{equation} \label{definitionbarz} \bar{Z}_{k,t,h} := \int_{\varepsilon(k-1)}^{\varepsilon k} f(t+h-s) \, dL_s, \quad k \in \mathbb{Z}, \end{equation}
where we suppress the dependence on $\varepsilon$ for notational convenience,
\begin{equation} \label{definitionbarrn} \bar{r}_{N,h,\varepsilon} := \frac{1}{N} \sum_{t=1}^N \sum_{k \neq k'} \bar{Z}_{k,t,0} \bar{Z}_{k',t,h}, \quad h \in \mathbb{N}_0, \end{equation}
and
\begin{equation} \label{definitionbardn} \bar{d}_{N,h,\varepsilon} := \frac{1}{N} \sum_{t=1}^N \sum_{k} \bar{Z}_{k,t,0}\bar{Z}_{k,t,h}, \quad h \in \mathbb{N}_0.\end{equation}
$\bar{r}_{N,h,\varepsilon}$ can be seen to converge unconditionally in  $L^2(\PP)$ like $r_{N,h,\varepsilon}$. By the next estimates, we see that $\bar{d}_{N,h,\varepsilon}$ converges in  $L^1(\PP)$  unconditionally, where we use Cauchy's equality for square integrable random variables in the first inequality and Cauchy's equality for square summable sequences in the second inequality and the fact that $a_k:=\sqrt{\int_{\varepsilon(k-1)}^{\varepsilon k } f(t-s)^2 \, ds}$ and $b_k:= \sqrt{\int_{\varepsilon(k-1)}^{\varepsilon k } f(t+h-s)^2 \, ds}$ for $k \in \mathbb{Z}$ define two sequences in $l^2(\mathbb{Z})$ and hence \[\sum_{k} |a_kb_k| \leq \sqrt{\sum_k a_k^2} \sqrt{\sum_k b_k^2}=\sum_{k} a_k^2\]
and consequently

\begin{eqnarray*}
& & \sum_{k} \E[|\bar{Z}_{k,t,0}\bar{Z}_{k,t,h}|]\\
&\leq & \sum_{k} \sqrt{\E[|\bar{Z}_{k,t,0}|^2]} \sqrt{\E[|\bar{Z}_{k,t,h}|^2]}\\
&=& \sigma^2 \sum_{k} \sqrt{\int_{\varepsilon(k-1)}^{\varepsilon k } f(t-s)^2 \, ds} \sqrt{\int_{\varepsilon(k-1)}^{\varepsilon k } f(t+h-s)^2 \, ds}\\
&\leq & \sigma^2 \sum_{k} \int_{\varepsilon(k-1)}^{\varepsilon k } f(t-s)^2 \, ds\\
&= & \sigma^2 \int_{- \infty}^{\infty}  f(t-s)^2 \, ds < \infty.
\end{eqnarray*}

Note that $X_{t+h}=\sum_{k = -\infty}^{\infty} \bar{Z}_{k,t,h}$ converges unconditionally in $L^2(\PP)$. Then by Lemma \ref{technicallemma1} we see that 
\begin{equation}
\hat{\gamma}_N(h)=\bar{r}_{N,h,\varepsilon}+\bar{d}_{N,h,\varepsilon} \label{decompositionautocoariance}
\end{equation}
is true as we showed the equality of (\ref{lefthandsidereaarrangement}) and (\ref{righthandsidereaarrangement}).
We want to prove the theorem by using Slutsky's lemma. As we have already seen
\[N^{1-2d}(r_{N,0,\varepsilon},\ldots,r_{N,H,\varepsilon}) \stackrel{d}{\rightarrow} C_d^2 \sigma^2 U_d(1) (1,\ldots,1),\]
it suffices to show for all $\delta>0$
\[\lim_{N \to \infty} \PP[|N^{1-2d} (\hat{\gamma}_N(h)-\gamma(h)-r_{N,h,\varepsilon})|>\delta]=0.\]

We split $\hat{\gamma}_N(h)-\gamma(h)-r_{N,h,\varepsilon}= (\bar{r}_{N,h,\varepsilon} - r_{N,h,\varepsilon}) + (\bar{d}_{N,h,\varepsilon} -\gamma(h))$.
Then we conclude the proof of Theorem \ref{Theoremfourthmoment} by using the next two lemmas.

\begin{lemma}\label{asymptoticrn} Let all assumptions of Theorem \ref{Theoremfourthmoment} apart from $\E[L_1^4] <\infty$ be fulfilled. Then for all $\delta>0$
\[\lim_{N \to \infty} \PP[|N^{1-2d}(\bar{r}_{N,h,\varepsilon} - r_{N,h,\varepsilon})|>\delta]=0.\]

\end{lemma}
\begin{proof}
We show that
\[\E[(N(\bar{r}_{N,h,\varepsilon} - r_{N,h,\varepsilon}))^2] = o(N^{4d})\]
and the claim then follows by Markov's inequality.
To this end, we mimic the proof of Lemma 5.5 in $\cite{HK}$ in the following.
One can see by the inequality $(a-b)^2\leq 2(a^2+b^2)$ that
\[(N(\bar{r}_{N,h,\varepsilon} - r_{N,h,\varepsilon}))^2 \leq 2 [S_1(N,\varepsilon) + S_2(N,\varepsilon)],\]
where we define $S_1(N,\varepsilon)$ and $S_2(N,\varepsilon)$ by
\[ S_1(N,\varepsilon):= \Big[ \sum_{t=1}^N \sum_{k \neq k'} f(t-\varepsilon k) Z_k ((f(t+h-\varepsilon k')Z_{k'} -\bar{Z}_{k',t,h})\Big]^2,\]
\[ S_2(N,\varepsilon):= \Big[ \sum_{t=1}^N \sum_{k \neq k'} (f(t-\varepsilon k)Z_k -\bar{Z}_{k,t,0})\bar{Z}_{k',t,h} \Big]^2.\]
Consequently, we evaluate $\E[S_1(N,\varepsilon)]$ and $\E[S_2(N,\varepsilon)]$.
We obtain
\begin{eqnarray}
&&\E[S_1(N,\varepsilon)] \label{ES1}\\
&=&\sum_{t,s=1}^{N} \sum_{k \neq k'} \sum_{i \neq i'} f(t-\varepsilon k) f(s - \varepsilon i) \E [Z_k (f(t+h-\varepsilon k')Z_{k'}-\bar{Z}_{k',t,h}) Z_i (f(s+h-\varepsilon i')Z_{i'}-\bar{Z}_{i',s,h})]. \nonumber\end{eqnarray}
Observe that $f(t+h-\varepsilon k)Z_{k}$ can be written as a stochastic integral:
\[f(t+h-\varepsilon k)Z_{k} = \int_{\varepsilon(k-1)}^{\varepsilon k} f(t+h-\varepsilon k) \, dL_s.\]
Hence by Itô's isometry, we have
\begin{eqnarray*}
\Var[(f(t+h-\varepsilon k)Z_{k}-\bar{Z}_{k,t,h})] &=& \sigma^2 \int_{\varepsilon(k-1)}^{\varepsilon k} (f(t+h-\varepsilon k)-f(t+h-s))^2 \, ds\\
&=& \sigma^2 \int_{0}^{\varepsilon} (f(t+h-\varepsilon k)-f(t+h-\varepsilon k+s))^2 \, ds.
\end{eqnarray*}

Define $\tilde{g}(x):= \sqrt{\int_{0}^{\varepsilon} (1-\frac{f(x+s)}{f(x)})^2 \, ds}$ for $x \in (0,\infty)$, where we assume for our calculations without loss of generality that $f$ does not vanish on $(0,\infty)$.

We obtain
\[\sqrt{\int_{0}^{\varepsilon} (f(x)-f(x+s))^2 \, ds}=x^{d-1} \tilde{g}(x) \frac{|f(x)|}{x^{d-1}} = O(x^{d-1} \tilde{g}(x)) \mbox{ as } x \to \infty\]
since $\lim_{x \to \infty} \frac{f(x)}{x^{d-1}}=C_d$.

We now show that $\lim_{x \to \infty} \tilde{g}(x)=0$. 
Let $\mu \in (0,C_d)$. Since $\lim_{x \to \infty} \frac{f(x)}{x^{d-1}}=C_d$, there is an $N_{\mu}$ such that $|\frac{f(x)}{x^{d-1}}-C_d|<\mu$ for $x \geq N_{\mu}$. Hence we find for $x \geq \max(N_{\mu},1)$ and $s \geq 0$
\begin{eqnarray*}
|1-\frac{f(x+s)}{f(x)}| & = & |\frac{f(x)-f(x+s)}{f(x)}|\\
& \leq & \Big| \frac{(C_d+\mu)x^{d-1}-(C_d-\mu)(x+s)^{d-1}}{(C_d-\mu)x^{d-1}} \Big|\\
&\leq & \Big| \frac{C_d}{C_d -\mu} (1-(\frac{x+s}{x})^{d-1})\Big| + 2 \mu.
\end{eqnarray*}
Hence
\[\limsup_{x \to \infty} \tilde{g}(x)^2 \leq \limsup_{x \to \infty} \int_{0}^{\varepsilon} (\Big| \frac{C_d}{C_d -\mu} (1-(\frac{x+s}{x})^{d-1})\Big|+ 2 \mu)^2 \, ds= \varepsilon 4 \mu^2\]
by Lebesgue's dominated convergence Theorem. Letting $\mu \to 0$, we obtain $\lim_{x \to \infty} \tilde{g}(x)=0$.

We can further replace $\tilde{g}$ by a bounded and decreasing function $g$  on $[1,\infty)$, e.g. by $g(x):=\sup_{y \geq x} \tilde{g}(y)$.

We define $v(t):=C( g(t) t^{d-1}1_{(1,\infty)}(t)+ 1_{[-\varepsilon,1]}(t))$ with $C$ large enough. Then

\[\sqrt{\int_{\varepsilon(k-1)}^{\varepsilon k} (f(t+h-\varepsilon k)-f(t+h-s))^2 \, ds} \leq v(t+h-\varepsilon k).\]

In the case $t+h-\varepsilon k < -\varepsilon$, the integral vanishes. $f$ is bounded by $K$ in the case $-\varepsilon \leq t+h-\varepsilon k < 1$.

All summands in (\ref{ES1}) vanish except for the cases (a) $k=i$ and $k'=i'$ and (b) $k=i'$ and $k'=i$.
By Cauchy's inequality, we get the estimate $\E[S_1(N,\varepsilon)] \leq \sigma^2 (E_{11} + E_{12})$ with
\[E_{11} := \sum_{t,s=1}^{N} \sum_{k \neq k'} |f(t-\varepsilon k) f(s-\varepsilon k)|  v(t+h-\varepsilon k') v(s+h-\varepsilon k')\]
and
\[E_{12} := \sum_{t,s=1}^{N} \sum_{k \neq k'} |f(t-\varepsilon k) f(s-\varepsilon k')|  v(t+h-\varepsilon k') v(s+h-\varepsilon k).\]
We consider here $E_{12}$. Similar calculations show that the same asymptotics also holds for $E_{11}$ and $\E[S_2(N,\varepsilon)]$.
We substitute $\varepsilon i = t -\varepsilon k$ and $\varepsilon i' = t+h - \varepsilon k'$. Then
\[|E_{12}|\leq \sum_{t,s=1}^{N} \sum_{i} |f(\varepsilon i)v(s-t+h+\varepsilon i)|\sum_{i'} |v(\varepsilon i') f(s-t-h+\varepsilon i')|.\]
Like in the proof of Lemma 5.5 in \cite{HK}, we denote with $E_{120}$ the summands where $s-t=0$, and $E_{12+}$ and $E_{12-}$ where $s-t>0$ or $s-t<0$ respectively. We obtain the upper estimates

\[|E_{120}|= N \sum_{i} |f(\varepsilon i)v(\varepsilon i+h)| \sum_{i'} |v(\varepsilon i') f(\varepsilon i'-h)| = O(N) \leq o(N^{4d}),\]
since $d>\frac{1}{4}$ and
\[|E_{12+}|\leq \sum_{n=1}^{N} (N-n) \sum_{i} |f(\varepsilon i)v(n+h+\varepsilon i)|\sum_{i'} |v(\varepsilon i') f(n-h+\varepsilon i')|.\]
By the integral convergence test, we obtain for $s-t=n \geq 1$
\begin{eqnarray*}
\sum_{i} |f(\varepsilon i)v(n+h+\varepsilon i)| &\leq & K v(n+h+\varepsilon)  + CK \int_{1}^{\infty} (\varepsilon x)^{d-1} (n+h+\varepsilon x)^{d-1}g(n+h+\varepsilon x) \, dx\\
&\leq& O(n^{d-1}g(n)) + CK n^{2d-1} \varepsilon^{2d-2} g(n) \int_{\frac{1}{n}}^{\infty} y^{d-1} (\frac{1}{\varepsilon} +y)^{d-1} \, dy
\end{eqnarray*}
and
\begin{eqnarray*}
\sum_{i'} |v(\varepsilon i') f(n-h+\varepsilon i')| &\leq& \sum_{i'=-1}^{(1+h)/\varepsilon} |v(\varepsilon i') f(n-h+\varepsilon i')| \\
&& + \sum_{i'=(1+h)/\varepsilon+1}^{\infty} CK g(\varepsilon i')(\varepsilon i')^{d-1} (n-h+\varepsilon i')^{d-1}\\
&\leq & O(n^{d-1}) + CK g(1) \int_{(1+h)/\varepsilon}^{\infty} (\varepsilon x)^{d-1} (n-h+\varepsilon x)^{d-1} \, dx\\ 
&\leq& O(n^{d-1}) + CK n^{2d-1} \varepsilon^{2d-2} g(1) \int_{\frac{1}{n}}^{\infty} y^{d-1} (\frac{1}{\varepsilon} +y)^{d-1} \, dy.
\end{eqnarray*}
Since $\varepsilon \leq 1$, we have
\[\int_{\frac{1}{n}}^{\infty} y^{d-1} (\frac{1}{\varepsilon} +y)^{d-1} \, dy  \leq \int_{0}^{\infty} y^{d-1} (1 +y)^{d-1} \, dy<\infty.\]
Hence
\[\sum_{i} |f(\varepsilon i)v(n+h+\varepsilon i)| = O(n^{2d-1}g(n))\]
and
\[\sum_{i'} |v(\varepsilon i') f(n-h+\varepsilon i')|=O(n^{2d-1}).\]
This gives
\[E_{12+} =O(N\sum_{n=1}^{N} n^{4d-2} g(n)) = o(N^{4d}) \mbox{ for } d>\frac{1}{4},\]
since
\[\lim_{N \to \infty} \frac{1}{N^{4d-1}} \sum_{j=1}^{N} j^{4d-2}g(j) \leq g(M)\lim_{N \to \infty} \frac{1}{N^{4d-1}} \sum_{j=M}^{N} j^{4d-2} = \frac{g(M)}{4d-1}\]
for all $M \in \mathbb{N}$ and letting $M \to \infty$, we see that indeed $ E_{12+} = o(N^{4d})$ is true.
Note that for $n \geq 1$
\[\sum_{i} |f(\varepsilon i)v(-n+h+\varepsilon i)| = \sum_{j} |f(n-h+\varepsilon j)v(\varepsilon j)|\]
and
\[\sum_{i'} |v(\varepsilon i') f(-n-h+\varepsilon i')|=\sum_{j'} |v(n+h+\varepsilon j') f(\varepsilon j')|.\]
Hence we can see by the calculuations above that $ E_{12-} =o(N^{4d})$ as well. \\
Note that if $d=\frac{1}{4}$, then $E_{12+}=O(N \log N)$ and if $d <\frac{1}{4}$, then $E_{12+}=O(N)$ and the same holds for $E_{12-}$, which we will use in Section 3.
\end{proof}

\begin{lemma}\label{asymptoticdn} Let all assumptions of Theorem \ref{Theoremfourthmoment} be fulfilled, especially $\E[L_1^4] <\infty$. Then for all $\delta>0$ 
\[\lim_{N \to \infty} \PP[|N^{1-2d}(\bar{d}_{N,h,\varepsilon}-\gamma(h))|>\delta]=0,\]
where $\bar{d}_{N,h,\varepsilon}$ is defined in (\ref{definitionbardn}).
\end{lemma}

\begin{proof}
Recall $\bar{Z}_{k,t,h} := \int_{\varepsilon (k-1)}^{\varepsilon k} f(t+h-s) \, dL_s$ for $k \in \mathbb{Z}$. Define
\[z_{k,t,0,h} := \sigma^2 \int_{\varepsilon(k-1)}^{\varepsilon k} f(t-s)f(t+h-s) \, ds = \E[\bar{Z}_{k,t,0}\bar{Z}_{k,t,h}], \quad k \in \mathbb{Z},\]
and
\[S_N:=\sum_{t=1}^{N} \xi_t, \quad N \in \mathbb{N},\] with 
\[\xi_t := \sum_{k} (\bar{Z}_{k,t,0}\bar{Z}_{k,t,h}-z_{k,t,0,h}), t \in \mathbb{N}.\]
Note that by the calculations below in the case $N=1$ we see that $\xi_t$ converges unconditionally in $L^2(\PP)$.
Hence by definition $N(\bar{d}_{N,h,\varepsilon}-\gamma(h)) =S_N$. We want to show
\[\lim_{N \to \infty} \PP[|N^{1-2d}(\bar{d}_{N,h,\varepsilon}-\gamma(h))|>\delta]=0.\]
We do this by showing that $\Var(S_N)=\E[S_N^2] = o(N^{4d})$ and using then Chebyshev's inequality. In fact we show $\E[S_N^2] = O(N)$ which we will also need in the proof of Lemma \ref{lemmaforbillingley2}. Since $\bar{Z}_{k,t,0}\bar{Z}_{k,t,h}-z_{k,t,0,h}$ and $\bar{Z}_{l,s,0}\bar{Z}_{l,s,h}-z_{l,s,0,h}$ have expectation zero and are independent for $k \neq l$, we obtain

\begin{eqnarray*}
\E[S_N^2]=\E[\sum_{t,s=1}^{N} \xi_t \xi_s] & = & \E[\sum_{t,s=1}^{N} (\sum_{k} \bar{Z}_{k,t,0}\bar{Z}_{k,t,h}-z_{k,t,0,h})(\sum_{l} \bar{Z}_{l,s,0}\bar{Z}_{l,s,h}-z_{l,s,0,h})]\\
&=&\sum_{t,s=1}^{N} \sum_{k} \E[ (\bar{Z}_{k,t,0}\bar{Z}_{k,t,h}-z_{k,t,0,h}) (\bar{Z}_{k,s,0}\bar{Z}_{k,s,h}-z_{k,s,0,h})]\\
&\leq&\sum_{t,s=1}^{N}  \sum_{k} \sqrt{\E[ (\bar{Z}_{k,t,0}\bar{Z}_{k,t,h}-z_{k,t,0,h})^2]}\sqrt{\E[(\bar{Z}_{k,s,0}\bar{Z}_{k,s,h}-z_{k,s,0,h})^2]}
\end{eqnarray*}
and further
\begin{eqnarray*}
&&\E[ (\bar{Z}_{k,t,0}\bar{Z}_{k,t,h}-z_{k,t,0,h})^2] \\
&=& \Var(\bar{Z}_{k,t,0}\bar{Z}_{k,t,h})\\
&\leq &\E[\bar{Z}_{k,t,0}^2\bar{Z}_{k,t,h}^2]\\
&\leq & \sqrt{\E[\bar{Z}_{k,t,0}^4]} \sqrt{\E[\bar{Z}_{k,t,h}^4]}\\
&\leq & \sqrt{[w(t-\varepsilon k)]^2} \sqrt{[w(t+h-\varepsilon k)]^2}  \leq   [w(t-\varepsilon k)]^2
\end{eqnarray*}
with $w(t):=C(t^{2d-2}1_{(1,\infty)}(t)+  1_{[-\varepsilon,1]}(t))$ with $C$ large enough and the following calculations:

By Lemma 3.2 in Cohen and Lindner \cite{CohenLindner}, we obtain
\[\E[\bar{Z}_{k,t,h}^4]=(\eta-3)\sigma^4 \int_{\varepsilon(k-1)}^{\varepsilon k} f(t+h-s)^4 \, ds + 3 \sigma^4 (\int_{\varepsilon (k-1)}^{\varepsilon k} f(t+h-s)^2 \, ds)^2.\]

If $t +h - \varepsilon k \geq 1$, then
\[\int^{\varepsilon k}_{\varepsilon (k-1)} f(t+h-s)^4 \, ds \leq \varepsilon K^4 (t+h-\varepsilon k)^{4d-4}\]
and
\[\int^{\varepsilon k}_{\varepsilon (k-1)} f(t+h-s)^2 \, ds \leq \varepsilon K^2 (t+h-\varepsilon k)^{2d-2}.\]
If $t +h - \varepsilon k <- \varepsilon$, then the integrals vanish. In the in-between case $-\varepsilon \leq t+h-\varepsilon k < 1$ the integrands are bounded by $K^2$ and $K^4$ respectively.

Hence we need to consider
$\sum_{s,t=1}^{N} \sum_k w(t-\varepsilon k) w(s- \varepsilon k)$. Substituting $\varepsilon i = t - \varepsilon k$, we get
$\sum_{s,t=1}^{N} \sum_i w(\varepsilon i)  w(\varepsilon i +s-t)$.
Adopting the notation $E_{0}$,$E_{+}$ and $E_{-}$  where we split the sum accordingly to the cases $n=s-t=0$, $n=s-t>0$ and $n=s-t<0$, we have $E_+=E_-$ since $\sum_i w(\varepsilon i)  w(\varepsilon i +n) =\sum_i w(\varepsilon i)  w(\varepsilon i -n)$ for $n \in \mathbb{N}$ and
\[E_0 = N \sum_{i} w(\varepsilon i)w(\varepsilon i) = O(N).\]
For $n \geq 1$, we get
\begin{eqnarray*}
&&\sum_i w(\varepsilon i) w(\varepsilon i +n)\\
& \leq & w(-\varepsilon)w(n-\varepsilon) + w(0)w(n) + C^2 \int_{1}^{\infty} (\varepsilon x)^{2d-2} (\varepsilon x + n)^{2d-2} \, dx\\
&=& O(n^{2d-2}) + n^{2d-2} C^2 \int_{1}^{\infty} (\varepsilon x)^{2d-2} (\frac{\varepsilon x}{n} + 1)^{2d-2} \, dx \\
&=& O(n^{2d-2}).
\end{eqnarray*}
Hence we obtain
\begin{eqnarray*}
E_+ &=& O(N \sum_{n=1}^{N} n^{2d-2})\\
&=& O (N^{2d})
\end{eqnarray*}
and
$\E[S_N^2] = O(N)$, since $d<\frac{1}{2}$.
\end{proof}

\section{Theorem for regularly varying tails}
In this section we show that the sample autocovariance of $X_t = \int_{-\infty}^{\infty} f(t-s) \, dL_s$ is asymptotically Rosenblatt or stable distributed (Theorem \ref{asymptoticalstable}), if $L_1$ is regularly varying with index $\alpha \in (2,4)$. This parallels Theorem 3.1 in \cite{HK}. We assume that $f(t)=0$ for $t \leq 0$, $f$ is bounded and $f(t) \sim C_d t^{d-1}$ as $t \to \infty$ with $d \in (0,\frac{1}{2})$ and $C_d>0$. 
Recall that a function $l:(0,\infty) \to (0,\infty)$ is called regularly varying with index $\rho$, if $\lim_{t \to \infty} \frac{l(tx)}{l(t)}=x^{\rho}$ for all $x>0$. We call a random variable $X$ regularly varying  with index $\alpha$, if the tail function $\bar{F}(x):=\PP[|X|>x]$ is regularly varying with index $-\alpha$. We say that $X$ fulfils a tail balance condition, if there is a $p \in [0,1]$ such that
\begin{equation} \lim_{x \to \infty} \frac{\PP[X>x]}{\PP[|X|>x]}=p.\label{tailbalancecondition} \end{equation}
If $X$ is symmetric, then $p$ equals $\frac{1}{2}$.

For $\alpha \in (0,2]$ we denote by $S_{\alpha}(\tau,\beta,\mu)$ an $\alpha$-stable distribution with $\tau \geq 0$ as scale parameter, $\beta \in [-1,1]$ as skewness and $\mu \in \R$ as location parameter, see (1.1.6), p. 9, in \cite{Samo}.

Let $(L_t)_{t \in \R}$ be a two-sided Lévy process. Assume that $L_1$ is regularly varying with index $\alpha \in (2,4)$ and fulfils the tail balance condition (\ref{tailbalancecondition}). Define
\[a_N := \inf\{y: \PP[|L_1|>y]<\frac{1}{N}\} \mbox{ and } b_{N}:= \E[L_{1}^2 1_{\{|L_1|\leq a_N\}}].\] Note that
\[a_N^2=\inf\{y^2: \PP[|L_1|^2>y^2]<\frac{1}{N}\}=\inf\{x: \PP[|L_1|^2>x]<\frac{1}{N}\}.\] Hence by Propositions 2.2.13 and 2.2.14 in \cite{EKM}, there is a stable distribution $S_{\frac{\alpha}{2}}(\tau,\beta,\mu)$ such that 
\begin{equation} \label{stablelimit}
\frac{1}{a_N^2} \sum_{t=1}^{N} ((L_t-L_{t-1})^2- b_{N}) \stackrel{d}{\rightarrow} S_{\frac{\alpha}{2}}(\tau,\beta,\mu).
\end{equation}
By Karamata's theorem, see Theorem 1.6.5 in Bingham et al. \cite{Bingham}, one can show that
\[\lim_{N \to \infty} \frac{N}{a_N^2}(\sigma^2 - b_N) = \frac{\alpha}{\alpha-2}.\]
Hence
\begin{equation} \label{stablelimit2}
\frac{1}{a_N^2} \sum_{t=1}^{N} ((L_t-L_{t-1})^2- \sigma^2) \stackrel{d}{\rightarrow} S_{\frac{\alpha}{2}}(\tau,\beta,\mu -\frac{\alpha}{\alpha-2}).
\end{equation}
Define two Lévy processes 
\begin{equation} \label{stableLevyprocess} (K_s)_{s \in [0,1]} \mbox{ with }K_1 \stackrel{d}{=}S_{\frac{\alpha}{2}}(\tau,\beta,\mu) \end{equation}
and
\begin{equation} \label{stableLevyprocess2} (M_s)_{s \in [0,1]} \mbox{ with }M_s := K_s - s \frac{\alpha}{\alpha-2}, \quad s \in [0,1].\end{equation}

\begin{theorem}\label{asymptoticalstable} Let $(L_t)_{t \in \mathbb{R}}$ be a two-sided Lévy process such that $L_1$ is symmetric about zero and has no Gaussian part and $\Var(L_1)=\sigma^2 \in (0,\infty)$. Assume that $L_1$ is regularly varying with index $\alpha \in (2,4)$. Define a continuous-time moving average process $(X_t)_{t \in \R}$ by \[X_t:=\int_{-\infty}^{\infty} f(t-s)  \, dL_s, \quad t \in \R,\] where we assume that $f(t)=0$ for $t \leq 0$, $f$ is bounded and $f(t) \sim C_d t^{d-1}$ as $t \to \infty$ with $d \in (0,\frac{1}{2})$ and $C_d>0$.  Let $H \in \mathbb{N}_0$. Define $f_m$ as in (\ref{fm}), i.e. \[f_m= \sum_{k=0}^{\infty} f(\varepsilon k) 1_{[k\varepsilon,(k+1)\varepsilon)},\] with $\varepsilon=\frac{1}{m}$. For $h \in \mathbb{N}_0$ define \[G_{h}(s):= \sum_{i=-\infty}^{\infty} f(i+s) f(i+h+s), \quad s \in [0,1],\]
\[G_{m,h}(s):= \sum_{i=-\infty}^{\infty} f_m(i+s) f_m(i+h+s), \quad s \in [0,1]\]
and
\[a_N := \inf\{y: \PP[|L_1|>y]<\frac{1}{N}\}.\]
Let $\hat{\gamma}(h)$ and $\gamma(h)$ denote the sample and actual covariance of $(X_t)_{t \in \mathbb{Z}}$ as in Section 2.

If $\frac{1}{\alpha} > d$ and $G_{m,h}$ converges in $L^{\frac{\alpha}{2}}([0,1])$ to $G_{h}$, then 
\[
\frac{N}{a_N^2} \Big(\hat{\gamma}_N(0)-\gamma(0),\ldots,\hat{\gamma}_N(H)-\gamma(H)\Big) \stackrel{d}{\rightarrow} \Big(\int_0^1 G_{0}(s) \, dM_s,\ldots, \int_0^1 G_{H}(s) \, dM_s \Big)\mbox{ as } N \to \infty,
\]
where the stochastic integrals with respect to $(M_s)_{s \in [0,1]}$ defined in (\ref{stableLevyprocess2}) are defined by convergence in probability, see e.g. Section 3.4 in \cite{Samo}. Observe that $G_h$ is bounded, hence it is in $L^{\frac{\alpha}{2}}([0,1])$ and the stochastic integrals are well-defined. \\
If $\frac{1}{\alpha} < d$, then 
\[N^{1-2d} (\hat{\gamma}_N(0)-\gamma(0),\ldots,\hat{\gamma}_N(H)-\gamma(H)) \stackrel{d}{\rightarrow} C_d^2 \sigma^2 U_d(1)(1,\ldots,1)\mbox{ as } N \to \infty,\]
where $U_d(1)$ is the marginal distribution of the Rosenblatt process at time $1$ defined in (\ref{Rosenblattprocess}).
\end{theorem}

\begin{remark} 1. The assumption that the sequence $G_{m,h}$ converges in $L^{\frac{\alpha}{2}}([0,1])$ to $G_{h}$ is for example fulfilled if $f$ is left-continuous. The sequence $f_m$ then converges pointwise to $f$. We can assume that $f_m$ is  bounded by $K(1_{[0,1]}(t)+t^{d-1}1_{(1,\infty)}(t))$ for $K$ large enough. Hence by Lebesgue's convergence theorem $G_{m,h}$ converges to $G_h$ pointwise, since $K^2(1+\sum_{t=1}^{\infty} t^{2(d-1)})$ is finite. Since the sequence $G_{m,h}$ is bounded by $K^2(1+\sum_{t=1}^{\infty} t^{2(d-1)})$, it converges by using Lebesgue's convergence theorem a second time in $L^{\alpha/2}([0,1])$ as well.
2. The assumption that $G_{m,h}$ converges in $L^{\frac{\alpha}{2}}([0,1])$ to $G_{h}$ can also be weakened  by assuming that $f$ coincides with a left-continuous function, say $\tilde{f}$, apart from a Lebesgue nullset. For then the processes $(X_t)_{t \in \R}$ based on $f$ and $\tilde{f}$ coincide, as do the quantities $G_h$ based on  $f$ and $\tilde{f}$ in $L^{\frac{\alpha}{2}}([0,1])$. Hence by proving the theorem for $\tilde{f}$, it also holds for $f$.\end{remark}

\begin{remark}
Note that Theorem \ref{asymptoticalstable} holds as well if we assume that $\E[L_1]=0$, $L_1$ fulfils the tail balance condition (\ref{tailbalancecondition}), $L$ has no Gaussian part and $\E[L_1^{\leq,a_N}]=0$ for all $N \in \mathbb{N}$ as we point out in Remark \ref{replacingcentringsequence} where also $L_1^{\leq,a_N}$ is defined.
\end{remark}

\begin{remark}
If $\frac{1}{\alpha}=d$, then we do not know completely what happens. $(a_N)_{N \in \mathbb{N}}$ is a regularly varying sequence with index $1/\alpha$, i.e. $a_N=N^{\frac{1}{\alpha}} l(N)$ for a slowly varying function $l$. If $\lim_{N \to \infty} l(N)=\infty$, then the limit is a stable distribution. If $\lim_{N \to \infty} l(N)=0$, then the limit is a Rosenblatt distribution. We do not know anything about the joint convergence of $r_{N,h,\varepsilon}$ and $d_{N,h,\varepsilon}$ defined in (\ref{definitionrn}) and in (\ref{definitiondn}). That is why we do not know what happens in the cases if $\lim_{N \to \infty} l(N) \in (0,\infty)$ or if the limit does not exist.
\end{remark}

The rest of this section is devoted to the proof of Theorem \ref{asymptoticalstable}.
As in section 2, we approximate $f$ by $f_m$ and consequently approximate $(X_t)_{t \in \mathbb{R}}$ by $(X_t^{(m)})_{t \in \R}$ defined by
\[X_t^{(m)}:=\int_{-\infty}^{\infty} f_m(t-s) \, dL_s = \sum_{k=0}^{\infty} f(\varepsilon k) (L_{t-\varepsilon k}-L_{t-\varepsilon(k+1)})=\sum_{k=0}^{\infty} f(\varepsilon k) Z_{mt-k}, \quad t \in \mathbb{Z}.\]

We split the autocovariance function again in diagonal and non-diagonal parts. The first step of the proof is to show that the squares of the smaller increments $(L_{t-\varepsilon k}-L_{t-\varepsilon(k+1)})$ are in the domain of attraction of the $m$-th convolution root of $S_{\frac{\alpha}{2}}(\tau,\beta,\mu)$. To this end, we need the following lemma:

\begin{lemma}\label{ratioofnormingsequences}
Let $(L_t)_{t \in \mathbb{R}}$ be a two-sided Lévy process. Assume that $L_1$ is regularly varying with index $\alpha \in (0,\infty)$ and fulfils the tail balance condition (\ref{tailbalancecondition}) with $p \in [0,1]$. Then $L_{\varepsilon}$ is also regularly varying with index $\alpha$ and fulfils (\ref{tailbalancecondition}) with $p$ as well. If we define the norming sequences $(a_N)_{N \in \mathbb{N}}$ and $(c_N)_{N \in \mathbb{N}}$ by \[a_N := \inf\{y: \PP[|L_1|>y]<\frac{1}{N}\} \mbox{ and } c_N := \inf\{y: \PP[|L_{\varepsilon}|>y]<\frac{1}{N}\},\] then \[\lim_{N \to \infty} \frac{a_N}{c_N} = m^{1/\alpha}.\] 
\end{lemma}
\begin{proof}
It is well-known, cf. e.g. Hult and Lindskog \cite{HultLindskog} Proposition 3.1, that an infinitely divisible distribution fulfilling a tail balance condition and its Lévy measure have the same behaviour as regularly varying measures. Note that this result originally goes back to Embrechts et al. \cite{Embrechts} for subexponential measures on $(0,\infty)$. Hence, since the Lévy measure of $L_{\varepsilon}$ is $\varepsilon$ times the Lévy measure of $L_1$, we see that $L_{\varepsilon}$ is regularly varying with index $\alpha$ and that it satisfies the same tail balance condition. Now define $G(x):=\PP[|L_1|>x]$ and $H(x):=\PP[|L_{\varepsilon}|>x]$. By the subexponentiality of the tails, see \cite{Fasen} section 1.1.1, we know that $\lim_{x \to \infty} \frac{G(x)}{H(x)}=m$. We set $U_1:=\frac{1}{G}$ and $U_2:=\frac{1}{H}$. $U_1$ and $U_2$ fulfil the conditions of Proposition 2.6 vi) of \cite{Resnick}. Hence we obtain $U_1^{\leftarrow}(x) \sim  m^{1/\alpha} U_2^{\leftarrow}(x)$ as $x \to \infty$, where $U_1^{\leftarrow}$ denotes the left-continuous inverse of $U_1$, which is defined by
\[U_1^{\leftarrow}(x) := \inf\{s:U_1(s) \geq x\} =\inf\{s: G(s) \leq \frac{1}{x}\}.\]
Hence because of $\frac{a_{N-1}}{a_N} \leq \frac{U_1^{\leftarrow}(N)}{a_N} \leq 1$, we see $a_N \sim U_1^{\leftarrow}(N)$ as $N \to \infty$. We see $c_N \sim U_2^{\leftarrow}(N)$ in the same manner.
This shows that $\lim_{N \to \infty} \frac{a_N}{c_N} = m^{1/\alpha}$.
\end{proof}

Now we have the following lemma:

\begin{lemma}\label{lemmaformixedterms}
Let $(L_t)_{t \in \R}$ be a two-sided Lévy process.  Assume that $L_1$ is regularly varying with index $\alpha \in (0,4)$ and fulfils the tail balance condition (\ref{tailbalancecondition}). Let $(a_N)_{N \in \mathbb{N}}$ be defined by $a_N := \inf\{y: \PP[|L_1|>y]<\frac{1}{N}\}$ and define $b_{N}:= \E[L_{1}^2 1_{\{|L_1|\leq a_N\}}]$ such that
\[\frac{1}{a_N^2} \sum_{t=1}^{N} ((L_t-L_{t-1})^2 -b_N) \stackrel{d}{\rightarrow} S_{\alpha/2}(\tau,\beta,\mu)\]
for constants $\tau, \beta$ and $\mu$.

Then
\[\frac{1}{a_N^2} \sum_{t=1}^{N} ((L_t-L_{t-\varepsilon})^2- \varepsilon b_{N}) \stackrel{d}{\rightarrow} S_{\frac{\alpha}{2}}(\varepsilon^{2/\alpha} \tau, \beta, \varepsilon \mu).\]
\end{lemma}

\begin{proof}
Define $c_N := \inf\{y: \PP[|L_{\varepsilon}|>y]<\frac{1}{N}\}$ and $d_N:=\E[L_{\varepsilon}^2 1_{\{|L_{\varepsilon}|\leq c_N\}}]$. Then as in (\ref{stablelimit}), we have
\[\frac{1}{c_N^2} \sum_{t=1}^{N} ((L_t-L_{t-\varepsilon})^2- d_{N}) \stackrel{d}{\rightarrow} T\]
for a stable law $T$. By Karamata's Theorem, we obtain $\lim_{N \to \infty} \frac{N}{a_N^2} (\sigma^2 - b_N) = \frac{\alpha}{\alpha -2}$ and $\lim_{N \to \infty} \frac{N}{c_N^2} (\varepsilon \sigma^2 - d_N) = \frac{\alpha}{\alpha -2}$. Hence by Lemma \ref{ratioofnormingsequences}, $\lim_{N \to \infty} \frac{N}{c_N^2} (\varepsilon \sigma^2 - \varepsilon b_N) = \varepsilon^{1-\frac{2}{\alpha}} \frac{\alpha}{\alpha -2}$ and hence in turn $\lim_{N \to \infty} \frac{N}{c_N^2} (\varepsilon b_N - d_N) = (1- \varepsilon^{1-\frac{2}{\alpha}}) \frac{\alpha}{\alpha -2}$. 
Hence we can replace $(d_N)_{N \in \mathbb{N}}$ by $(\varepsilon b_N)_{N \in \mathbb{N}}$ and still obtain a stable limit. By Lemma \ref{ratioofnormingsequences} we can replace $(c_N)_{N \in \mathbb{N}}$ by $(a_N)_{N \in \mathbb{N}}$ and obtain a stable limit as well. We conclude that
\[\frac{1}{a_N^2} \sum_{t=1}^{N} ((L_t-L_{t-\varepsilon})^2 - \varepsilon b_N) \stackrel{d}{\rightarrow} S\]
for a stable law $S$. We now show that $S=S_{\frac{\alpha}{2}}(\varepsilon^{2/\alpha} \tau, \beta, \varepsilon \mu)$.

We split the normed partial sums in the following manner:
\begin{eqnarray*}
\frac{1}{a_N^2} \sum_{t=1}^{N} ((L_t-L_{t-1})^2 -b_N)& = & \frac{1}{a_N^2} \sum_{t=1}^{N} \Big(\Big( \sum_{i=1}^{m} (L_{t+\varepsilon i}-L_{t+\varepsilon (i-1)})\Big)^2 - b_N \Big)\\
&=&\underbrace{\sum_{i=1}^{m} \underbrace{\frac{1}{a_N^2}  \sum_{t=1}^{N} \Big((L_{t+\varepsilon i}-L_{t+\varepsilon (i-1)})^2 - \varepsilon b_N \Big)}_{\rightarrow  \mbox{stable law } S}}_{\rightarrow S^{m \ast}} \\
&&+ \sum_{i,j=1 i \neq j}^{m} \underbrace{\frac{1}{a_N^2}  \sum_{t=1}^{N} (L_{t+\varepsilon i}-L_{t+\varepsilon (i-1)})(L_{t+\varepsilon j}-L_{t+\varepsilon(j-1)})}_{\stackrel{\PP}{\rightarrow} 0},
\end{eqnarray*}
where the convergences are justified as follows:
since we sum up random variables of independent sequences, the sum converges towards the convolution of the distributional limits. By Lemma 4.1 of \cite{Jessen}, $|(L_{t+\varepsilon i}-L_{t+\varepsilon (i-1)})(L_{t+\varepsilon j}-L_{t+\varepsilon(j-1)})|$ is regularly varying with index $\alpha$ for $i \neq j$, hence \[\frac{1}{a_N^2}  \sum_{t=1}^{N} |(L_{t+\varepsilon i}-L_{t+\varepsilon (i-1)})(L_{t+\varepsilon j}-L_{t+\varepsilon(j-1)})|, \quad i \neq j,\] converges in probability to zero by Lemma \ref{ratioofnormingsequences} and hence so does \[\frac{1}{a_N^2}  \sum_{t=1}^{N} (L_{t+\varepsilon i}-L_{t+\varepsilon (i-1)})(L_{t+\varepsilon j}-L_{t+\varepsilon(j-1)}), \quad i \neq j.\] Now it is obvious (see e.g. \cite{Samo}, Property 1.2.1, p.10) that $S= S_{\frac{\alpha}{2}}(\varepsilon^{2/\alpha} \tau, \beta, \varepsilon \mu)$. 
\end{proof}

We now show that the diagonal parts of the autocovariance function of the approximated process $(X_t^{(m)})_{t \in \R}$ converge to a stable law expressed as a stochastic integral with respect to a stable Lévy process.

\begin{lemma} \label{diagonalparttostable}
Let $(L_t)_{t \in \R}$ be a two-sided Lévy process. Assume that $L_1$ is regularly varying with index $\alpha \in (2,4)$ and fulfils the tail balance condition (\ref{tailbalancecondition}). Let $f$ and $f_m$ be as in Theorem \ref{asymptoticalstable}. Define \[a_N := \inf\{y: \PP[|L_1|>y]<\frac{1}{N}\} \mbox{ and } b_{N}:= \E[L_{1}^2 1_{\{|L_1|\leq a_N\}}].\]
Define \[G_{m,h}(s):= \sum_{i=-\infty}^{\infty} f_m(i+s) f_m(i+h+s), \quad s \in [0,1]\]
and
\begin{equation} \label{definitiondn} d_{N,h,\varepsilon} := \frac{1}{N} \sum_{t=1}^{N} \sum_{i=0}^{\infty} f(\varepsilon i) f(\varepsilon i + h) ((Z_{mt-i})^2 - \varepsilon b_{N}),\end{equation}
where $(Z_i)_{i \in \mathbb{Z}}$ was defined in (\ref{increments}). Then $d_{N,h,\varepsilon}$ converges absolutely almost surely and in $L^1(\PP)$.
Further, \[\frac{N}{a_N^{2}} (d_{N,0,\varepsilon},\ldots,d_{N,H,\varepsilon}) \stackrel{d}{\rightarrow} (\int_0^1 G_{m,0}(s) \, dK_s,\ldots, \int_0^1 G_{m,H}(s) \, dK_s),\]
where $(K_s)_{s \in [0,1]}$ is defined in (\ref{stableLevyprocess}). Observe that $G_{m,h}$ is bounded, hence it is in $L^{\frac{\alpha}{2}}([0,1])$ and the stochastic integrals are well-defined.
\end{lemma}

\begin{proof} The almost sure absolute convergence and convergence in $L^1(\PP)$ of $d_{N,h,\varepsilon}$ are clear.
By Lemma \ref{lemmaformixedterms}, we know that
\[\frac{1}{a_N^2} \sum_{t=1}^{N} ((L_t-L_{t-\varepsilon})^2- \varepsilon b_{N}) \stackrel{d}{\rightarrow} S_{\frac{\alpha}{2}}(\varepsilon^{2/\alpha} \tau, \beta, \varepsilon \mu).\]

By rearranging one sees that
\[d_{N,h,\varepsilon} =\frac{1}{N} \sum_{t=1}^N \sum_{j \in \{0,\ldots,m-1\}} \sum_{i=0}^{\infty} f(i+ \varepsilon j) f(i+h+ \varepsilon j)((L_{t-i-\varepsilon j}-L_{t-i-\varepsilon(j-1)})^2-\varepsilon b_{N}).\]

By Theorem 4.1 of Davis and Resnick \cite{DavisResnick1985} and using the technique used in the proof of Lemma 5.1 in \cite{HK}, we obtain with
\[d_{N,h,\varepsilon,j}:= \frac{1}{N} \sum_{t=1}^N  \sum_{i=0}^{\infty} f(i+ \varepsilon j) f(i+h+ \varepsilon j)((L_{t-i-\varepsilon j}-L_{t-i-\varepsilon(j-1)})^2-\varepsilon b_{N})\]
that
\[\frac{N}{a_N^2} \Big(d_{N,0,\varepsilon,j},\ldots,d_{N,H,\varepsilon,j}\Big)\stackrel{d}{\rightarrow}\Big(G_{m,0}(\varepsilon j),\ldots, G_{m,H}(\varepsilon j)\Big) S_{\frac{\alpha}{2}}(\varepsilon^{2/\alpha} \tau, \beta, \varepsilon \mu)\]
for each $j\in \{0,\ldots,m-1\}$.  Since the sequences are independent for different $j$, the convolution of the  limits equals the limit of the sums and the claimed convergence follows.  Observe that we use the fact that $G_{m,h}$ is an equidistant step function.
\end{proof}

The convergence of the non-diagonal parts of the autocovariance function to the Rosenblatt distribution was already established in section 2. Hence we can state the following lemma:

\begin{lemma} \label{lemmaforbillingley1} Let the assumptions of Lemma \ref{diagonalparttostable} be fulfilled. Define $r_{N,h,\varepsilon}$ as in (\ref{definitionrn}) and $d_{N,h,\varepsilon}$ as in (\ref{definitiondn}).
If $\frac{1}{\alpha} > d$, then
\[\frac{N}{a_N^2} ((d_{N,0,\varepsilon}  + r_{N,0,\varepsilon}),\ldots,(d_{N,H,\varepsilon}  + r_{N,H,\varepsilon})) \stackrel{d}{\rightarrow} (\int_0^1 G_{m,0}(s) \, dK_s,\ldots, \int_0^1 G_{m,H}(s) \, dK_s).\]
If $\frac{1}{\alpha} < d$, then
\[N^{1-2d} ((d_{N,0,\varepsilon}  + r_{N,0,\varepsilon}),\ldots,(d_{N,H,\varepsilon}  + r_{N,H,\varepsilon})) \stackrel{d}{\rightarrow} C_d^2 \sigma^2 U_d(1) (1,\ldots,1) .\]
\end{lemma}
\begin{proof}
Note that $a_N^2$ is regularly varying with index $\frac{2}{\alpha}$. Hence $N^{2d}=o(a_N^2)$ if $\frac{1}{\alpha} > d$ and $a_N^2 = o(N^{2d})$ if $\frac{1}{\alpha} < d$. The lemma then follows by Slutsky's lemma using Lemma \ref{nondiagonalparttoRosenblatt} and Lemma \ref{diagonalparttostable}.
\end{proof}

\begin{remark} \label{replacingcentringsequence} Now let us assume that the Lévy process is symmetric and has no Gaussian part.
We decompose the Lévy process into two independent Lévy processes: let $L^{\leq,a_N}$ have the Lévy measure of $L$ restricted to $[-a_N,a_N]$ and $L^{>,a_N}$ have the Lévy measure of $L$ restricted to $[-a_N,a_N]^c$. Note that $L^{\leq,a_N}$ has the variance $\Var(L^{\leq,a_N}_1)= \int_{-a_N}^{a_N} x^2 \nu(dx)$, see Example 25.12 in \cite{Sato}, hence $\E[(L^{\leq,a_N})^2]=\int_{-a_N}^{a_N} x^2 \nu(dx) + \E[L_1^{\leq,a_N}]^2= \int_{-a_N}^{a_N} x^2 \nu(dx)$, while $\E[(L_1 1_{\{|L_1|\leq a_N\}})^2]= \int_{-a_N}^{a_N} x^2 \mu (dx)$ where $\mu$ denotes the distribution of $L_1$ and $\nu$ denotes its Lévy measure. Of course $\mu$ and $\nu$ are in general not equal but $\E[(L_1)^2]=\int_{-\infty}^{\infty} x^2 \mu(dx) = \int_{-\infty}^{\infty} x^2 \nu(dx) +\E[L_1]^2=\int_{-\infty}^{\infty} x^2 \nu(dx)$ is true, see Example 25.12 in \cite{Sato}. We know additionally that both $\mu$ and $\nu$ have the same tail behaviour, i.e. they are both regularly varying with the same index and the same tail balance condition, see for example Hult and Lindskog \cite{HultLindskog} Proposition 3.1. Hence by Karamata's theorem, by the tail equivalence of $\mu$ and $\nu$ and by the symmetry of $\mu$ and $\nu$,
\begin{eqnarray*} \lim_{N \to \infty} \frac{N}{a_N^2}(\sigma^2 - \int_{-a_N}^{a_N} x^2 \, \nu(dx))&=& \lim_{N \to \infty} \frac{N}{a_N^2} 2\int_{a_N}^{\infty} x^2 \, \nu(dx)\\
=\frac{\alpha}{\alpha-2}&=&\lim_{N \to \infty} \frac{N}{a_N^2}(\sigma^2 - \int_{-a_N}^{a_N} x^2 \, \mu(dx)).\end{eqnarray*}
Hence we can replace the centring sequence $b_{N}:= \E[L_{1}^2 1_{\{|L_1|\leq a_N\}}]$ by $\tilde{b}_{N}:= \E[(L_{1}^{\leq,a_N})^2]$ without changing the limit in (\ref{stablelimit}). Note that $\lim_{N \to \infty} \frac{N}{a_N^2}(b_N - \tilde{b}_N)=0$.

Note finally that this works as well if we assume that $\E[L_1]=0$, $L$ has no Gaussian part and $\E[L_1^{\leq,a_N}]=0$ for all $N \in \mathbb{N}$.
\end{remark}

\begin{lemma} \label{lemmaforbillingley2} Let the assumptions of Lemma \ref{diagonalparttostable} be fulfilled. Further assume that $L_1$ is symmetric about zero and has no Gaussian part. If $\frac{1}{\alpha} > d$, then for all $\delta>0$
\[\lim_{\varepsilon \to 0} \limsup_{N \to \infty} \PP[\Big|\frac{N}{a_N^2} \Big(\hat{\gamma}_N(h)-b_N \int_{-\infty}^{\infty} f(s) f(s+h) \, ds-(d_{N,h,\varepsilon}+r_{N,h,\varepsilon})\Big)\Big|>\delta]=0.\]
If $\frac{1}{\alpha} < d$, then for all $\delta>0$
\[\lim_{N \to \infty} \PP[\Big|N^{1-2d} \Big(\hat{\gamma}_N(h)-b_N \int_{-\infty}^{\infty} f(s) f(s+h) \, ds-(d_{N,h,\varepsilon}+r_{N,h,\varepsilon})\Big)\Big|>\delta]=0.\]
\end{lemma}
\begin{proof}
We first show the first claim. Note that by (\ref{decompositionautocoariance})
\begin{eqnarray*}
&& \hat{\gamma}_N(h)-b_N \int_{-\infty}^{\infty} f(s) f(s+h) \, ds-(d_{N,h,\varepsilon}+r_{N,h,\varepsilon})\\
&=&\bar{d}_{N,h,\varepsilon}-b_N \int_{-\infty}^{\infty} f(s) f(s+h) \, ds-d_{N,h,\varepsilon}-(r_{N,h,\varepsilon}-\bar{r}_{N,h,\varepsilon}).
\end{eqnarray*}

We denote $Z_{i}^{\leq,a_N}:=L_{\varepsilon i}^{\leq,a_N}-L_{\varepsilon(i-1)}^{\leq,a_N}$ and $Z_{i}^{>,a_N}:=L_{\varepsilon i}^{>,a_N}-L_{\varepsilon(i-1)}^{>,a_N}$ in analogy to (\ref{increments}), where we use the notation of Remark \ref{replacingcentringsequence}. In the same manner we define in analogy to (\ref{definitionbarz}), (\ref{definitionbardn}) and (\ref{definitiondn})
\begin{equation} \bar{Z}_{k,t,h}^{\leq,a_N} := \int_{\varepsilon(k-1)}^{\varepsilon k} f(t+h-s) \, dL_s^{\leq,a_N}, \quad k \in \mathbb{Z}, \end{equation}
\begin{equation} \bar{Z}_{k,t,h}^{>,a_N} := \int_{\varepsilon(k-1)}^{\varepsilon k} f(t+h-s) \, dL_s^{>,a_N}, \quad k \in \mathbb{Z}, \end{equation}
\begin{equation} \bar{d}_{N,h,\varepsilon}^{\leq,a_N} := \frac{1}{N} \sum_{t=1}^N \sum_{k} \bar{Z}_{k,t,0}^{\leq,a_N}\bar{Z}_{k,t,h}^{\leq,a_N}, \quad h \in \mathbb{N}_0\end{equation}
and
\begin{equation}  d_{N,h,\varepsilon}^{\leq,a_N} := \frac{1}{N} \sum_{t=1}^{N} \sum_{i=0}^{\infty} f(\varepsilon i) f(\varepsilon i + h) ((Z_{mt-i}^{\leq,a_N})^2 - \varepsilon \tilde{b}_{N}).\end{equation}
$\bar{d}_{N,h,\varepsilon}^{\leq,a_N}$ and $d_{N,h,\varepsilon}^{\leq,a_N}$ can be seen to converge unconditionally as in (\ref{definitionbardn}) and (\ref{definitiondn}), where we use the fact that $\E[L_1^{\leq,a_N}]=0$ for all $N \in \mathbb{N}$ by our symmetry assumption.

We consider the upper estimate
\begin{eqnarray*}
&&\PP[|\frac{N}{a_N^2} (\bar{d}_{N,h,\varepsilon}-b_N (\int_{-\infty}^{\infty} f(s) f(s+h) \, ds)-d_{N,h,\varepsilon}-(r_{N,h,\varepsilon}-\bar{r}_{N,h,\varepsilon}))|>\delta]\\
&\leq &\PP[|\frac{N}{a_N^2} (\bar{d}_{N,h,\varepsilon}^{\leq,a_N}-\tilde{b}_N (\int_{-\infty}^{\infty} f(s) f(s+h) \, ds))|>\frac{\delta}{7}]\\
&+&\PP[|\frac{N}{a_N^2} (d_{N,h,\varepsilon}^{\leq,a_N})|>\frac{\delta}{7}]\\
&+&\PP[|\frac{N}{a_N^2} (\bar{r}_{N,h,\varepsilon}-r_{N,h,\varepsilon})|>\frac{\delta}{7}]\\
&+&\PP[|\frac{N}{a_N^2} (\frac{1}{N} \sum_{t=1}^{N} \sum_{i=0}^{\infty} f(\varepsilon i) f(\varepsilon i+ h) (Z_{mt-i}^{>,a_N})^2-\frac{1}{N} \sum_{t=1}^N \sum_{k} (\bar{Z}_{k,t,0}^{>,a_N})(\bar{Z}_{k,t,h}^{>,a_N}))|>\frac{\delta}{7}]\\
&+&\PP[|\frac{N}{a_N^2} (\frac{1}{N} \sum_{t=1}^{N} \sum_{i=0}^{\infty} f(\varepsilon i) f(\varepsilon i + h) (Z_{mt-i}^{\leq,a_N})(Z_{mt-i}^{>,a_N})-\frac{1}{N} \sum_{t=1}^N \sum_{k} (\bar{Z}_{k,t,0}^{\leq,a_N})(\bar{Z}_{k,t,h}^{>,a_N}))|>\frac{\delta}{7}]\\
&+&\PP[|\frac{N}{a_N^2} (\frac{1}{N} \sum_{t=1}^{N} \sum_{i=0}^{\infty} f(\varepsilon i) f(\varepsilon i+ h) (Z_{mt-i}^{\leq,a_N})(Z_{mt-i}^{>,a_N})-\frac{1}{N} \sum_{t=1}^N \sum_{k} (\bar{Z}_{k,t,0}^{>,a_N})(\bar{Z}_{k,t,h}^{\leq,a_N}))|>\frac{\delta}{7}]\\
&+&\PP[|\frac{N}{a_N^2} (b_N-\tilde{b}_N)(\int_{-\infty}^{\infty} f(s) f(s+h) \, ds  - \varepsilon \sum_{i=0}^{\infty} f(\varepsilon i) f(\varepsilon i+ h))|>\frac{\delta}{7}].
\end{eqnarray*}
The last term vanishes for $N$ large enough since  $\lim_{N \to \infty} \frac{N}{a_N^2}(b_N - \tilde{b}_N)=0$ by Remark \ref{replacingcentringsequence}. $\Var(N(\bar{d}_{N,h,\varepsilon}^{\leq,a_N}-\tilde{b}_N (\int_{-\infty}^{\infty} f(s) f(s+h) \, ds)))$ is in $O(NE[(L_1^{\leq,a_N})^2])\leq O(N)$ by the calculations of Lemma \ref{asymptoticdn}. Hence \[\limsup_{N \to \infty} \PP[|\frac{N}{a_N^2} (\bar{d}_{N,h,\varepsilon}^{\leq,a_N}-\tilde{b}_N (\int_{-\infty}^{\infty} f(s) f(s+h) \, ds))|>\frac{\delta}{7}]=0,\] since $a_N^2$ is a regularly varying sequence with index $2/\alpha$ and $\alpha<4$.
This also applies to the second term by using $f_m$ instead of $f$ as the corresponding function. The third term is negligible by the calculations in Lemma \ref{asymptoticrn}, which show that $\E[(N(\bar{r}_{N,h,\varepsilon} - r_{N,h,\varepsilon}))^2] = o(a_N^4)$
if $d<\frac{1}{\alpha}$, if $d \neq \frac{1}{4}$. (For $d<\frac{1}{4}$ this even holds for all $\alpha$.) In the case $d=\frac{1}{4}$, we additionally need that $\lim_{N \to \infty} a_N^{-4} N \ln N = 0$, since $\E[(N(\bar{r}_{N,h,\varepsilon} - r_{N,h,\varepsilon}))^2] = O(N \log N)$ holds in this case. Since $a_N=N^{1/\alpha} l(N)$ for a slowly varying function $l$, we have $a_N^{-4}=N^{1-4/\alpha}  \ln N l(N)^{-4}$. Since $\ln N l(N)^{-4}$ is slowly varying as well, $\lim_{N \to \infty} a_N^{-4} N \ln N = 0$ holds.\\
Now we consider the fourth term. To this end, define
\[\xi_t:= \sum_{k} (f(t-\varepsilon k) f(t+h-\varepsilon k) (Z_k^{>,a_N})^2-\bar{Z}_{k,t,0}^{>,a_N}\bar{Z}_{k,t,h}^{>,a_N}).\]
$\xi_t$ can be seen to converge unconditionally in $L^1(\PP)$.
Then
\[\sum_{t=1}^{N} \sum_{i=0}^{\infty} f(\varepsilon i) f(\varepsilon i+ h) (Z_{mt-i}^{>,a_N})^2- \sum_{t=1}^N \sum_{k} (\bar{Z}_{k,t,0}^{>,a_N})(\bar{Z}_{k,t,h}^{>,a_N})= \sum_{t=1}^N \xi_t.\]
We split the $\xi_t$ in two parts:
\begin{eqnarray*}
&&\sum_{k} (f(t-\varepsilon k) f(t+h-\varepsilon k) (Z_k^{>a_N})^2 - \bar{Z}^{>a_N}_{k,t,0}\bar{Z}^{>a_N}_{k,t,h})\\
&=&\sum_{k} ((\int_{\varepsilon (k-1)}^{\varepsilon k} f(t-\varepsilon k) \, dL^{>a_N}_s)(\int_{\varepsilon (k-1)}^{\varepsilon k} f(t+h-\varepsilon k) \, dL^{>a_N}_s)\\
&&-(\int_{\varepsilon (k-1)}^{\varepsilon k} f(t-s) \, dL^{>a_N}_s)(\int_{\varepsilon (k-1)}^{\varepsilon k} f(t+h-s) \, dL^{>a_N}_s))\\
&=&\sum_{k} ((\int_{\varepsilon (k-1)}^{\varepsilon k} f(t-\varepsilon k) \, dL^{>a_N}_s)(\int_{\varepsilon (k-1)}^{\varepsilon k} (f(t+h-\varepsilon k)-f(t+h-s)) \, dL^{>a_N}_s)\\
&&+(\int_{\varepsilon (k-1)}^{\varepsilon k} (f(t-\varepsilon k)-f(t-s)) \, dL^{>a_N}_s)(\int_{\varepsilon (k-1)}^{\varepsilon k} f(t+h-s) \, dL^{>a_N}_s))\\
&=:& A_t + B_t.
\end{eqnarray*}
We show that $\limsup_{N \to \infty} \E[\frac{1}{a_N^2} \sum_{t=1}^{N} |A_t|] = O(\varepsilon)$ as $\varepsilon \to 0$. The same argument applies to $B_t$. The calculations also show that the series defining $A_t$ and $B_t$ converge a.s. absolutely and unconditionally in $L^1(\PP)$ and hence they are well-defined.

We define a third Lévy process $|L^{>a_N}|$ which is defined  by $|L^{>a_N}|_t := \sum_{0 < s \leq t} |\Delta L^{>a_N}_s|$ for $t \geq 0$ and $|L^{>a_N}|_t := -\sum_{t <s <0} |\Delta L^{>a_N}_s|$ for $t < 0$. Define \[u(t):=C(t^{d-1}1_{(1,\infty)}(t)+  1_{[-\varepsilon,1]}(t))\] with $C$ large enough. We obtain

\begin{eqnarray*}
&&\E[|(\int_{\varepsilon (k-1)}^{\varepsilon k} f(t-\varepsilon k) \, dL^{>a_N}_s)(\int_{\varepsilon (k-1)}^{\varepsilon k} (f(t+h-\varepsilon k)-f(t+h-s)) \, dL^{>a_N}_s)|]\\
&\leq& \E[(\int_{\varepsilon (k-1)}^{\varepsilon k} |f(t-\varepsilon k)| \, d|L^{>a_N}|_s)(\int_{\varepsilon (k-1)}^{\varepsilon k} |(f(t+h-\varepsilon k)-f(t+h-s))| \, d|L^{>a_N}|_s)]\\
&\leq& \varepsilon^2 |f(t-\varepsilon k)| u(t+h-\varepsilon k) \E[(|L^{>a_N}_1|)^2]\\
&\leq&  \varepsilon^2 |f(t-\varepsilon k)| u(t-\varepsilon k) \E[(|L^{>a_N}|_1)^2],
\end{eqnarray*}
since $u(t-\varepsilon k) \geq u(t+h-\varepsilon k)$ if $t-\varepsilon k \geq -\varepsilon$ and $f(t-\varepsilon k)=0$ otherwise.
Note that we use $|f(t+h-\varepsilon k)-f(t+h-s)| \leq u(t+h-\varepsilon k)$ by the triangle inequality for $C$ large enough. Note further
\begin{eqnarray*}
\sum_k  |f(t-\varepsilon k)| u(t-\varepsilon k)&=& \sum_{i=1}^{\frac{1}{\varepsilon}} |f(\varepsilon i)u(\varepsilon i)| + \sum_{1+\frac{1}{\varepsilon}}^{\infty} |f(\varepsilon i)u(\varepsilon i)|\\
&\leq & KC\frac{1}{\varepsilon} +  KC \int_{\frac{1}{\varepsilon}}^{\infty} (\varepsilon x)^{2d-2} \, dx\\
& =&  KC\frac{1}{\varepsilon} + KC \frac{1}{\varepsilon} \int_{1}^{\infty} y^{2d-2} \, dy = O(\frac{1}{\varepsilon}) \mbox{ as } \varepsilon \to 0.
\end{eqnarray*}
By Karamata's theorem, see Theorem 1.6.5 in Bingham et al. \cite{Bingham}, \[\lim_{N \to \infty} \frac{N}{a_N^2} 2\int_{a_N}^{\infty} x^2 \, \nu (dx)\] exists and is finite. Also by Karamata's theorem,  $\lim_{N \to \infty} \frac{N}{a_N} 2 \int_{a_N}^{\infty} x \, \nu (dx)$ exists and is finite, hence $\lim_{N \to \infty} \frac{N}{a_N^2} (2 \int_{a_N}^{\infty} x \, \nu (dx))^2=0$. Thus
\[\lim_{N \to \infty} \frac{N}{a_N^2} \E[|L^{>a_N}|_1^2]= \lim_{N \to \infty} \frac{N}{a_N^2} [2\int_{a_N}^{\infty} x^2 \, \nu (dx) + (2\int_{a_N}^{\infty} x \, \nu (dx))^2]\]
exists and is finite. Hence $\limsup_{N \to \infty} \E[\frac{1}{a_N^2} \sum_{t=1}^{N} |A_t|] = O(\varepsilon)$ as $\varepsilon \to 0$. In the same fashion one can show that $\limsup_{N \to \infty} \E[\frac{1}{a_N^2} \sum_{t=1}^{N} |B_t|] = O(\varepsilon)$ as $\varepsilon \to 0$. By Markov's inequality, this gives convergence of the fourth term to 0 when letting first $N \to \infty$ and then $\varepsilon \to 0$.

Finally, we consider the fifth and sixth term. They are dealt with in the same manner. We use the same reasoning as for the fourth term and consider $\sum_{t=1}^N \xi_t$ where we replace $(Z_{mt-i}^{>,a_N})^2$ by $(Z_{mt-i}^{>,a_N})(Z_{mt-i}^{\leq ,a_N})$ and replace the first or second factor of $(\bar{Z}_{k,t,0}^{>,a_N})(\bar{Z}_{k,t,h}^{>,a_N})$. We then consider  $A_t$ and see that it is negligible by the following calculations:

By using the independence of $L^{>a_N}$ and $L^{\leq a_N}$, Jensen's inequality and Itô's isometry, we obtain
\begin{eqnarray*}
&&\E[|(\int_{\varepsilon (k-1)}^{\varepsilon k} f(t-\varepsilon k) \, dL^{\leq a_N}_s)(\int_{\varepsilon (k-1)}^{\varepsilon k} (f(t+h-\varepsilon k)-f(t+h-s)) \, dL^{>a_N}_s)|]\\
&\leq &\E[|(\int_{\varepsilon (k-1)}^{\varepsilon k} f(t-\varepsilon k) \, dL^{\leq a_N}_s)|] \E[|(\int_{\varepsilon (k-1)}^{\varepsilon k} (f(t+h-\varepsilon k)-f(t+h-s)) \, dL^{>a_N}_s)|]\\
&\leq &\sqrt{\E[((\int_{\varepsilon (k-1)}^{\varepsilon k} f(t-\varepsilon k) \, dL^{\leq a_N}_s))^2]} \E[|(\int_{\varepsilon (k-1)}^{\varepsilon k} |(f(t+h-\varepsilon k)-f(t+h-s))| \, d|L^{>a_N}|_s)|]\\
&\leq &\sqrt{(\int_{\varepsilon (k-1)}^{\varepsilon k} f(t-\varepsilon k)^2 \, ds)} \, \sigma u(t+h-\varepsilon k) \varepsilon \E[|L^{>a_N}|_1]\\
&\leq & \varepsilon^{\frac{3}{2}} |f(t-\varepsilon k)| \sigma u(t-\varepsilon k)  \E[|L^{>a_N}|_1].
\end{eqnarray*}
By Karamata's theorem, $\lim_{N \to \infty} \frac{N}{a_N} \E[|L^{>a_N}|_1]$ exists and is finite, hence \[\lim_{N \to \infty} \frac{N}{a_N^2} \E[|L^{>a_N}|_1]=0\]
and in turn $\limsup_{N \to \infty} \E[\frac{1}{a_N^2} \sum_{t=1}^{N} |A_t|] = 0$ for all $\varepsilon > 0$. In this case similar calculations show that  $\limsup_{N \to \infty} \E[\frac{1}{a_N^2} \sum_{t=1}^{N} |B_t|] =  0$ as well.

The second claim follows from the calculations for the first case and the fact that $a_N^2= o(N^{2d})$ if $d>\frac{1}{\alpha}$.
\end{proof}

Returning to the proof Theorem \ref{asymptoticalstable},
We first conclude the proof for the case $\frac{1}{\alpha} > d$: 
We show
\begin{eqnarray*}
&&\frac{N}{a_N^2} \Big(\hat{\gamma}_N(0)-b_N \int_{-\infty}^{\infty} f(s)f(s) \, ds,\ldots,\hat{\gamma}_N(H)-b_N \int_{-\infty}^{\infty} f(s)f(s+H) \, ds\Big)\\
&\stackrel{d}{\rightarrow}& \Big(\int_0^1 G_{0}(s) \, dK_s,\ldots, \int_0^1 G_{H}(s) \, dK_s\Big),\mbox{ as } N \to \infty.
\end{eqnarray*}
Since the sequence $G_{m,h}$ converges in $L^{\frac{\alpha}{2}}([0,1])$ to $G_{h}$, it follows that $\int_0^1 G_{m,h}(s) \, dK_s \stackrel{d}{\rightarrow} \int_0^1 G_h(s) \, dK_s$, see Proposition 3.5.1 in \cite{Samo}. Hence for the one-dimensional result by Theorem 3.2 in \cite{Billingsley} together with Lemma \ref{lemmaforbillingley1}, it suffices to check for all $\delta>0$
\[\lim_{\varepsilon \to 0} \limsup_{N \to \infty} \PP[|\frac{N}{a_N^2} (\hat{\gamma}_N(h)-b_N (\int_{-\infty}^{\infty} f(s) f(s+h) \, ds)-(d_{N,h,\varepsilon}+r_{N,h,\varepsilon}))|>\delta]=0,\]
which has been proved in Lemma \ref{lemmaforbillingley2}.
The multidimensional results also follows by the simple fact that a vector converges in probability if its components converge in probability.

By Karamata's theorem, see Theorem 1.6.5 in Bingham et al., we see
\begin{eqnarray}\lim_{N \to \infty} \frac{N}{a_N^2}(\sigma^2 - b_N) &=& \lim_{N \to \infty} \frac{N}{a_N^2}(\sigma^2 - \int_{-a_N}^{a_N} x^2 \, \mu(dx)) \nonumber\\
&=& \lim_{N \to \infty} \frac{N}{a_N^2} 2\int_{a_N}^{\infty} x^2 \, \mu(dx)= \frac{\alpha}{\alpha-2}. \label{alphaadjustment}\end{eqnarray}
By (\ref{alphaadjustment}), we have
\begin{eqnarray*}
\lim_{N \to \infty} \frac{N}{a_N^2}\Big(b_N \int_{-\infty}^{\infty} f(s)f(s+h) \, ds - \gamma(h)\Big) & = &\lim_{N \to \infty} \frac{N}{a_N^2}(b_N-\sigma^2)\int_{-\infty}^{\infty} f(s)f(s+h) \, ds\\
& = & -\frac{\alpha}{\alpha-2} \int_{-\infty}^{\infty} f(s)f(s+h) \, ds.
\end{eqnarray*}
Hence we conclude
\begin{eqnarray*}
&&\frac{N}{a_N^2} \Big(\hat{\gamma}_N(0)-\gamma(0),\ldots,\hat{\gamma}_N(H)-\gamma(H)\Big)\\
&\stackrel{d}{\rightarrow}& (\int_0^1 G_{0}(s) \, dK_s - \frac{\alpha}{\alpha-2}\int_{-\infty}^{\infty} f(s)f(s) \, ds,\ldots, \int_0^1 G_{H}(s) \, dK_s - \frac{\alpha}{\alpha-2}\int_{-\infty}^{\infty} f(s)f(s+H) \, ds).
\end{eqnarray*}
Note that $\int_{-\infty}^{\infty} f(s)f(s+h) \, ds = \int_{0}^{1} G_{h}(s) \, ds$.
Hence we finally conclude
\[
\frac{N}{a_N^2} \Big(\hat{\gamma}_N(0)-\gamma(0),\ldots,\hat{\gamma}_N(H)-\gamma(H)\Big) \stackrel{d}{\rightarrow} \Big(\int_0^1 G_{0}(s) \, dM_s,\ldots, \int_0^1 G_{H}(s) \, dM_s \Big)\mbox{ as } N \to \infty.
\]

We now consider the case $\frac{1}{\alpha} < d$: Since $\lim_{N \to \infty} \frac{N}{a_N^2}(\sigma^2 - b_N) =\frac{\alpha}{\alpha-2}$, we have \[\lim_{N \to \infty} N^{1-2d}(\sigma^2 - b_N) =0.\] Hence we can replace $\gamma(h)$ without loss of generality by $b_N (\int_{-\infty}^{\infty} f(s) f(s+h) \, ds)$ without changing the limit.
Hence for the one-dimensional result by Slutsky's Lemma together with Lemma \ref{lemmaforbillingley1}, we have to check that for all $\delta>0$
\[\lim_{N \to \infty} \PP[|N^{1-2d} ( \hat{\gamma}_N(h)-b_N (\int_{-\infty}^{\infty} f(s) f(s+h) \, ds)-(d_{N,h,\varepsilon}+r_{N,h,\varepsilon}))|>\delta]=0,\]
which is the statement of Lemma \ref{lemmaforbillingley2}.

\section{Remarks}

\begin{remark}
If we assume that $\E[L_1^4]= \eta \sigma^4<\infty$ and that $f(t)=0$ for $t \leq 0$, $f$ is bounded and $f(t) \sim C_d t^{d-1}$ as $t \to \infty$ with $d \in (0,\frac{1}{4})$ and $C_d>0$, then the conditions of Theorem 3.5 (a) in Cohen and Lindner \cite{CohenLindner} are fulfilled, i.e. the sample autocovariance is asymptotically normal distributed. More precisely,
\[\sqrt{N} (\hat{\gamma}_N(0)-\gamma(0),\ldots, \hat{\gamma}_N(H)-\gamma(H)) \stackrel{d}{\rightarrow} N(0,V),\]
where the covariance matrix $V=(v_{pq})_{p,q =0,\ldots,H}$ is given by
\[v_{pq}= (\eta-3)\sigma^4 \int_0^1 G_p (u) G_q(u) \, du + \sum_{k=-\infty}^{\infty} [\gamma(k)\gamma(k-p+q)+ \gamma(k+q)\gamma(k-p)]\]
and $G_p$ is as defined in Theorem \ref{asymptoticalstable}.
This corresponds to Theorem 3.5 (a) in \cite{HK}.
\end{remark}
\begin{proof}
We have to check the assumptions of Proposition 3.1 and Theorem  3.5 (a) in \cite{CohenLindner}. It is easy to see that $f \in L^2(\R) \cap L^4(\R)$, since $|f(t)|\leq K \max(1,t^{d-1})$ for a constant $K$. We next check that the function $[0,1] \rightarrow \R, u \mapsto \sum_{k=-\infty}^{\infty}  f(u+k)^2$ is in $L^2([0,1])$, which is in our notation the function $G_0$. Observe that $G_0$ is bounded by $K^2(1 + \sum_{k=1}^{\infty} k^{2d-2})<\infty$. Hence it even is in  $L^{\infty}([0,1])$.
We need not check (3.3) in \cite{CohenLindner} since (3.11) in \cite{CohenLindner} is stronger than (3.3).
Finally, we turn to (3.11) in \cite{CohenLindner}, i.e.
\begin{equation} \sum_{h=1}^{\infty} (\int_{-\infty}^{\infty} |f(s)||f(s+h)| \, ds)^2<\infty \label{conditionCL}.\end{equation}
By our assumptions on the kernel function, we obtain
\begin{eqnarray*}
\int_{-\infty}^{\infty} |f(s)||f(s+h)| \, ds \leq K^2 \int_0^{\infty} t^{d-1} (t+h)^{d-1} \, dt. \end{eqnarray*}
By substitution, we obtain for $h>0$
\[\int_0^{\infty} t^{d-1} (t+h)^{d-1} \, dt = h^{2d-1} \int_0^{\infty} s^{d-1} (s+1)^{d-1} \, ds.\]
Then (\ref{conditionCL}) is an immediate consequence of $d < \frac{1}{4}$ and the result follows from Theorem 3.5 in \cite{CohenLindner}.\end{proof}

\begin{remark}
If we assume that $(L_t)_{t \in \R}$ is a two-sided Brownian motion with $\E[L_1^2]=\sigma^2$ and that $f(t)=0$ for $t \leq 0$, $f$ is bounded and $f(t) \sim C_d t^{d-1}$ as $t \to \infty$ with $d =\frac{1}{4}$ and $C_d>0$, then 
\[\sqrt{\frac{N} {\log N}} (\hat{\gamma}_N(0)-\gamma(0),\ldots, \hat{\gamma}_N(H)-\gamma(H)) \stackrel{d}{\rightarrow} 2 C_d^2 \sigma^2 Z (1,\ldots,1)\mbox{ as } N \to \infty,\]
where $Z$ is a standard normal distributed random variable.
This follows from the proof of Theorem 4 (ii) in the article by  Hosking \cite{Hosking}. The proof shows the convergence by the method of cumulants and we have decided not to explore this case further for  non-Brownian Lévy processes with finite fourth moments. 
\end{remark}

The following definition of a FICARMA process goes back to Brockwell, see \cite{Brockwell} and \cite{BrockwellMarquardt}.
\begin{defi}
Let $a(z)=z^p+a_1 z^{p-1} + \cdots + a_p$ and $b(z)=b_0+b_1 z + \cdots b_q z^q$ be polynomials with real coefficients with $a_p \neq 0$, $b_q\neq 0$ and $q<p$. Let $d \in (0,\frac{1}{2})$. Let $(L_t)_{t \in \mathbb{R}}$ be a two-sided Lévy process with $\E[L_1]=0$ and $\Var(L_1)=\sigma^2 \in (0,\infty)$. If the roots of $a(z)$ all have negative real parts, then a FICARMA($p$,$d$,$q$) process $(X_t)_{t \in \R}$ is defined by $X_t := \int_{-\infty}^{\infty} f(t-s) \, dL_s$, where the kernel function $f$ is defined by $f(t)=0$ for $t \leq 0$ and
\[f(t)= \frac{1}{2\pi} \int_{-\infty}^{\infty} e^{i t \lambda} (i \lambda)^{-d} \frac{b(i\lambda)}{a(i \lambda)} \, d \lambda \quad \mbox{for } t >0.\] 
\end{defi}
The next proposition shows that the kernel function of a FICARMA process fulfils the assumption on $f$ in the assumptions of Theorem \ref{Theoremfourthmoment} and Theorem \ref{asymptoticalstable} if $b_0 \neq 0$. Hence depending on the Lévy process, one of these theorems can be applied.
\begin{propo}
The kernel function of a FICARMA process fulfils $f(t) \sim \frac{t^{d-1}}{\Gamma(d)}\cdot \frac{b(0)}{a(0)}$ as $t \to \infty$. Further, $f$ is infinitely often differentiable  on $\R^+$ and $\lim_{t \searrow 0} f(t)=0$, hence $f$ is bounded.
\end{propo}
\begin{proof}
Brockwell \cite{Brockwell} shows that $f(t) \sim \frac{t^{d-1}}{\Gamma(d)}\cdot \frac{b(0)}{a(0)}$ by refering to Theorem 37.1 on page 254 in \cite{Doetsch} and states this as equation (4.6) in \cite{Brockwell}. He rewrites the kernel function as $\frac{1}{2\pi i} \int_{-i\infty}^{i \infty} e^{t z} z^{-d} \frac{b(z)}{a(z)} \, d z$. By the discussion on page 250 in \cite{Doetsch}, the path through the origin can be replaced by a path $\mathfrak{W}$ with angle $\psi$ as in figure 37 on page 250 in \cite{Doetsch}. Since $a(z)$ has only finitely many roots, we find $\psi>\frac{\pi}{2}$ small enough such that all singularities lie on the left-hand side of $\mathfrak{W}$. We have $\lim_{|z| \to \infty} z^{-d} \frac{b(z)}{a(z)} =0$ since $|z^{-d}| = |z|^{-d}$ and $q < p$. Thus the conditions of Theorem 37.1 in \cite{Doetsch} are fulfilled. To show that it is differentiable, we now assume without loss of generality that we use the path $\mathfrak{W}$, hence
\[f(t)=\frac{1}{2\pi i} \int_{\mathfrak{W}} e^{t z} z^{-d} \frac{b(z)}{a(z)} \, d z.\]
By Theorem 36.1 in \cite{Doetsch}, $f$ is analytic on the right-hand side of $\mathfrak{W}$ with derivative
\[f'(t)=\frac{1}{2\pi i} \int_{\mathfrak{W}} e^{t z} z^{-d+1}\frac{b(z)}{a(z)} \, d z.\]
Note that the choice $\mathfrak{W}$ depends on $t$.

Since $f$ is analytic on the right-hand side of $\mathfrak{W}$, it is especially infinitely often differentiable on the positive real line. The kernel function can be equivalently expressed as $f(t)=\int_0^t g(t-u) \frac{u^{d-1}}{\Gamma(d)} \, du$, see \cite{Brockwell} equation (4.4), which is the Riemann-Liouville fractional integral of the kernel $g$ of a CARMA process with polynomials $a(z)$ and $b(z)$. Hence by equation (2.5) on page 46 in \cite{Miller}, $\lim_{t \searrow 0} f(t)= 0$.
\end{proof}

Theorem \ref{Theoremfourthmoment} and Theorem \ref{asymptoticalstable} gave limit theorems for the sample autocovariance function of $(X_t)_{t \in \mathbb{Z}}$.  Using the delta-method, it is easy to obtain limit theorems for the sample autocorrelation defined by $\hat{\rho}_N(h):=\frac{\hat{\gamma}_N(h)}{\hat{\gamma}_N(0)}$ for $h \in \mathbb{N}_0$. The autocorrelation function is denoted by $\rho(h):=\frac{\gamma(h)}{\gamma(0)}$ for $h \in \mathbb{N}_0$.
\begin{coro}
If
\[N^{1-2d} (\hat{\gamma}_N(0)-\gamma(0),\ldots, \hat{\gamma}_N(H)-\gamma(H)) \stackrel{d}{\rightarrow} C_d^2 \sigma^2 U_d(1) (1,\ldots,1)\mbox{ as } N \to \infty,\]
where $U_d(1)$ is the marginal distribution of the Rosenblatt process at time $1$ defined in (\ref{Rosenblattprocess}),
then
\[N^{1-2d}(\hat{\rho}_N(h)-\rho(h))\stackrel{d}{\rightarrow} C_d^2 \sigma^2 U_d(1) \frac{1-\rho(h)}{\gamma(0)}\mbox{ as } N \to \infty.\]
If
\[ \frac{N}{a_N^2} (\hat{\gamma}_N(0)-\gamma(0),\ldots,\hat{\gamma}_N(H)-\gamma(H)) \stackrel{d}{\rightarrow} (\int_0^1 G_{0}(s) \, dM_s,\ldots, \int_0^1 G_{H}(s) \, dM_s) \mbox{ as } N \to \infty,
\]
then

\[\frac{N}{a_N^2}(\hat{\rho}_N(h)-\rho(h)) \stackrel{d}{\rightarrow} \frac{1}{\gamma(0)}\int_0^1 (G_{h}(s)-\rho(h) G_{0}(s))  \, dM_s \mbox{ as } N \to \infty.
\]

\end{coro}
\begin{proof}
This follows from the delta-method, see Theorem 3.1.  in \cite{vanderVaart}, with the function $\varphi(x,y)=\frac{y}{x}$.
\end{proof}

We conclude this section by considering fractional Lévy noise as in \cite{CohenLindner}.
One way of defining a fractional Lévy process $(M_t)_{t \in \R}$ is to set
\[M_t:=  \frac{1}{\Gamma(d+1)} \int_{- \infty}^{\infty} ((t-s)_+^d-(-s)_+^d) \, dL_s, \quad t \in \R,\]
where $d \in (0,\frac{1}{2})$, $(L_t)_{t \in \R}$ is a two-sided Lévy process with $\E[L_1]=0$ and $\Var(L_1)<\infty$, see \cite{Marquardt}. One obtains fractional Lévy noise $(Z_t)_{t \in \mathbb{Z}}$ by defining
\[Z_t:= M_t-M_{t-1}, \quad t \in \mathbb{Z}.\]
Hence $Z_t= \frac{1}{\Gamma(d+1)} \int_{- \infty}^{\infty} ((t-s)_+^d-(t-1-s)_+^d) \, dL_s$. By the mean value theorem, one obtains that $\frac{1}{\Gamma(d+1)}(t_+^d-(t-1)_+^d) \sim \frac{d}{\Gamma(d+1)}t^{d-1}$ as $t \to \infty$.
Cohen and Lindner derive that
\[\hat{d}:= \frac{1}{2} \frac{\log (\hat{\rho}_N(1)+1)}{\log 2}\]
is a strongly consistent estimator for $d$, see (4.2) in \cite{CohenLindner}, where $\hat{\rho}$ is the sample autocorrelation of the fractional Lévy noise. Proposition 4.1 in \cite{CohenLindner} states that $\hat{d}$ is asymptotically normal distributed, if $d \in (0,\frac{1}{4})$ and $\E[L_1^4]<\infty$. By the delta-method and the last corollary, we can see that for example in the case $d \in (\frac{1}{4},\frac{1}{2})$ and $\E[L_1^4]<\infty$, $\hat{d}$ is asymptotically Rosenblatt distributed. This complements the results in \cite{CohenLindner}.

\section*{Acknowledgement}
I am grateful for the support by Deutsche Forschungsgemeinschaft Grant LI 1026/4-2.


\begin{thebibliography}{99}
\bibitem{Billingsley} P. Billingsley \textsl{Convergence of Probability Measures}, John Wiley, 1999
\bibitem{Bingham} N. H. Bingham, C. M. Goldie, J. L. Teugels \textsl{Regular Variation}, Cambridge University Press, 1987
\bibitem{Brockwell} P. J. Brockwell Representations of continuous-time ARMA processes, J. Appl. Probab. \textbf{41A}, 375-382, 2004
\bibitem{BrockwellDavis} P. J. Brockwell, R. A. Davis \textsl{Time Series: Theory and Methods} Second Edition Springer-Verlag, New York 2006
\bibitem{BrockwellMarquardt} P. J. Brockwell, T. Marquardt Lévy-driven and fractionally integrated ARMA processes with continuous time parameter, Statistica Sinica \textbf{15}(2005), 477-494
\bibitem{CohenLindner} S. Cohen, A. Lindner A central limit theorem for the sample autocorrelations of a Lévy driven continuous time moving average process, Journal of Statistical Planning and Inference \text{143}, 1295-1306, 2013
\bibitem{DavisResnick1985} R. Davis, S. Resnick Limit Theory for Moving Averages of Random Variables with Regularly Varying Tail Probabilities, The Annals of Probability, \textbf{13},(1) 179-195, 1985
\bibitem{Doetsch} G. Doetsch \textsl{Introduction to the Theory and Application of the Laplace Transform}, Springer Verlag, Berlin 1974
\bibitem{Embrechts}  P.Embrechts, C. Goldie, N. Veraverbeeke \textsl{Subexponentiality and infinite divisibility}, Z. Wahrschein. Verw. Geb. \textbf{49}, 335-347
\bibitem{EKM} P. Embrechts, C. Klüppelberg, T. Mikosch \textsl{Modelling Extremal Events}, Springer Verlag, Springer Verlag, Berlin 1997
\bibitem{Fasen} V. Fasen \textsl{Extremes of Lévy Driven Moving Average Processes with Applications in Finance}, PhD thesis, TU München, 2004
\bibitem{GKS} L. Giraitis, H. L. Koul, D. Surgailis \textsl{Large Sampe Inference for Long Memory Processes} Imperial College Press, London 2012
\bibitem{HK} L. Horváth, P. Kokoszka Sample autocovariances of long-memory time series Bernoulli \textbf{2}, 2008, 405-418
\bibitem{Hosking} J. R. M. Hosking Asymptotic distributions of the sample mean, autocovariances and autocorrelations of long-memory time series, Journal of Econometrics \textbf{73}, 261-268, 1996
\bibitem{HultLindskog} H. Hult, F. Lindskog On regular variation for infinitely divisible random vectors and additive processes, Adv. Appl. Probl. \textbf{38}, 134-148, 2006
\bibitem{Jessen} A. H. Jessen, T. Mikosch \textsl{Regularly Varying Functions} Publications de l'institut mathématique  Nouvelle série, tome 80(94) 171-192 2006
\bibitem{Marquardt} T. Marquardt \textsl{Fractional Lévy processes with an application to long memory moving average processes}, Bernoulli Volume 12, Number 6, 1099-1126, 2006
\bibitem{Miller} K. S. Miller, B. Ross \textsl{An Introduction to the Fractional Calculus and Fractional Differential Equations}, John Wiley, New York 1993
\bibitem{Resnick} S. Resnick \textsl{Heavy-Tail Phenomena}, Springer 2006
\bibitem{Samo} G. Samorodnitsky, M. Taqqu \textsl{Stable Non-Gaussian Random Processes}, Chapman \& Hall, 1994
\bibitem{Sato} K. Sato, \textsl{Lévy Processes and Infinitely Divisible Distributions}, Cambridge Studies in Advanced Mathematics, 1999
\bibitem{vanderVaart} A. W. van der Vaart \textsl{Asymptotic Statistics}, Cambridge Press, 1998
\end{thebibliography}
\end{document}